\documentclass{amsart}
\usepackage{amscd,amssymb,amsopn,amsmath,amsthm,graphics,amsfonts,accents,enumerate,verbatim,calc}
\usepackage[dvips]{graphicx}
\usepackage[colorlinks=true,linkcolor=red,citecolor=blue]{hyperref}
\usepackage[all]{xy}

\usepackage{tikz}

\calclayout

\newcommand{\rt}{\rightarrow}
\newcommand{\lrt}{\longrightarrow}

\newcommand{\st}{\stackrel}

\newcommand{\La}{\Lambda}

\newcommand{\CA}{\mathcal{A} }

\newcommand{\CG}{\mathcal{G} }

\newcommand{\CN}{\mathcal{N} }
\newcommand{\CP}{\mathcal{P} }
\newcommand{\CQ}{\mathcal{Q} }

\newcommand{\CS}{\mathcal{S} }
\newcommand{\CT}{\mathcal{T} }
\newcommand{\CX}{\mathcal{X} }
\newcommand{\CY}{\mathcal{Y} }

\newcommand{\CB}{\mathcal{B} }

\newcommand{\Mod}{{\rm{Mod\mbox{-}}}}

\newcommand{\mmod}{{\rm{{mod\mbox{-}}}}}
\newcommand{\mmodd}{{\rm{mod}}_0\mbox{-}}

\newcommand{\Prj}{{\rm{Prj}\mbox{-}}}
\newcommand{\prj}{{\rm{prj}\mbox{-}}}

\newcommand{\GPrj}{{\GP}\mbox{-}}
\newcommand{\Gprj}{{\Gp\mbox{-}}}
\newcommand{\GInj}{{\GI \mbox{-}}}

\newcommand{\op}{{\rm{op}}}

\newcommand{\GP}{{\rm{GPrj}}}
\newcommand{\Gp}{{\rm{Gprj}}}
\newcommand{\GI}{{\rm{GInj}}}

\newcommand{\Coker}{{\rm{Coker}}}
\newcommand{\Ker}{{\rm{Ker}}}

\newcommand{\Tor}{{\rm{Tor}}}
\newcommand{\Hom}{{\rm{Hom}}}
\newcommand{\Ext}{{\rm{Ext}}}

\theoremstyle{plain}
\newtheorem{theorem}{Theorem}[section]
\newtheorem{corollary}[theorem]{Corollary}
\newtheorem{lemma}[theorem]{Lemma}

\newtheorem{proposition}[theorem]{Proposition}

\newtheorem{notation}[theorem]{Notation}

\theoremstyle{definition}
\newtheorem{definition}[theorem]{Definition}
\newtheorem{example}[theorem]{Example}

\newtheorem{construction}[theorem]{Construction}

\newtheorem{remark}[theorem]{Remark}

\theoremstyle{plain}

\theoremstyle{definition}

\numberwithin{equation}{section}

\begin{document}

\title[ Auslander-Reiten duality for subcategories]{ Auslander-Reiten duality for subcategories}

\author[Rasool Hafezi ]{Rasool Hafezi }

\address{School of Mathematics, Institute for Research in Fundamental Sciences (IPM), P.O.Box: 19395-5746, Tehran, Iran}
\email{hafezi@ipm.ir}




\begin{abstract}

Auslander-Reiten duality for module categories is generalized to some  sufficiently nice subcategories. In particular, our consideration works for $\mathcal{P}^{< \infty}(\Lambda)$, the subcategory consisting of finitely generated modules with finite projective dimension over an artin algebra $\Lambda$, and also, the subcategory of Gorenstein projective modules of $\rm{mod}\mbox{-} \Lambda$, denoted by $\rm{Gprj}\mbox{-} \Lambda.$ In this paper, we give a method to compute the Auslander-Reiten translation in $\mathcal{P}^{< \infty}(\Lambda)$ whenever $\Lambda$ is a $1$-Gorenstein algebra. In addition, we characterize when the Auslander-Reiten translation in $\rm{Gprj}\mbox{-} \Lambda$ is the first syzygy and provide many algebras having such property.  

\end{abstract}

\maketitle
\section{Introduction}
The classical Auslander-Reiten duality for modules over an artin algebra $\La$ says that
\[D\underline{\Hom}_{\La}(X, Y)\simeq \Ext_{\La}^1(Y, D\rm{Tr}(X))\]
for $M$ and $N$ in $\mmod \La$. Here, $D$ denotes the duality over a fixed commutative artin  ground ring $k$ and $\rm{Tr}$ denotes the transpose construction. It was established by Auslander and Reiten \cite{AR2}  and later generalized to modules
over arbitrary rings \cite{Au}. Recently, Auslander-Reiten duality for module categories has been  generalized by Henning Krause to Grothendieck abelian categories that have a sufficient supply of finitely presented objects \cite{K}. Furthermore,   an Auslander-Reiten formula for the setting of functorially finite subcategories coming from cotilting modules was given by  K. Erdmann et al \cite{EMM}.

The aim of this paper is to establish  an Auslander-Reiten duality for  sufficiently nice subcategories of $\mmod \La.$ More precisely, we prove the following results in this paper. 

\begin{theorem}\label{IndTheorem}
	\hspace{-1mm}
\begin{itemize}	
	\item[$(i)$]Let $\CX \subseteq \mmod \La$ be resolving-coresolving,  functorially finite and closed under direct summand. Then $\CX$ has  Auslander-Reiten duality, i.e, there exist an equivalence functor $\tau_{\CX}:\underline{\CX} \rt \overline{\CX}$  such that for any $X$ and $Y$ in $\CX$, we have the following natural isomorphism in both variables;
	\[D\underline{\Hom}_{\La}(X, Y)\simeq \Ext_{\La}^1(Y, \tau_{\CX}(X)).\]
	
	\item[$(ii)$] Assume that $\rm{Gprj}\mbox{-} \La$ is contravariently finite subcategory in $\mmod \La.$ Then $\rm{Gprj}\mbox{-} \La$ has Auslander-Reiten duality,  i.e. there exist  an equivalence functor $\tau_{\mathcal{G}}:\underline{\rm{Gprj}}\mbox{-} \La \rt \underline{\rm{Gprj}}\mbox{-} \La$  such that for any Gorenstein projective modules $G$ and $G'$, we have the following natural isomorphism in both variables;
	\[D\underline{\Hom}_{\La}(G, G')\simeq \Ext_{\La}^1(G', \tau_{\mathcal{G}}(G)).\]
	
\end{itemize}
\end{theorem}

Note that here stabe-Hom (modulo projective modules) and  $\Ext^1$ are taken in the original category, that is, $\mmod \La.$ The uniquely determined up to isomorphism equivalence for the corresponding subcategory in above theorem is called relative Auslander-Reiten translation functor in the  subcategory. We refer to Theorem \ref{AR-duality} and Theorem \ref{tranGprj} for the proof of the theorem.
 
In Example \ref{Examplefirst}, we provide many subcategories satisfying the requested  condition in the first part of theorem. One of the such important   examples is the subcategory of modules with finite projective dimension over artin algebra $\La$,  $\CP^{<\infty}(\La)$. In fact, the subcategory of $\mmod \La$ consisting of  Gorenstein projectove modules, $\rm{Gprj}\mbox{-} \La,$ see \ref{gproj} for definition, is related to $\CP^{<\infty}(\La)$ via a cotorsion theory whenever $\La$ is a Gorenstein algebra. So, in some special cases we can see some connection between subcategories appeared in the first part and the second part of the theorem.

 It is worth noting that there is a close relation between the existence of almost split sequence and the existence of the Auslander-Reiten duality. For abelian case, we refer to  a paper by Lenzing and Zuazua \cite{ LZ}. Recently, in \cite{LNP} a local version of the main result stated in \cite{LZ} was given in the setting of exact categories. There are several works to investigate the existence of almost split sequences in subcategories, but, not much works about the existence of Auslander-Reiten dulaity for subcatgories. This paper is an attempt for this missing part.
  
To prove the first part of above theorem, we in more general setting extend an equivalence due to Claus M.Ringel and Pu Zhang \cite{RZ} for the case of $\La=K[x]/(x^n)$, and recently generalized  by \"{O}gmundur Eiriksson \cite{E} to self injective algebras of finite type.

\begin{theorem}\label{Zeroequivalence}
Let $\CA$ be an abelian category with enough projectives.
Let $\CX$ be a subcategory of $\CA$ including $\rm{Proj} \mbox{-}\CA$, being contravariantly finite and closed under kernel of epimorphisms. Consider the full subcategory $\mathcal{V}$ of $S_{\CX}(\CA)$ formed by  finite direct sums of   objects in the form of $(X \st{\rm{Id}}\rt X )$ or $ (0 \rt X )$, that $X$ runs through  all of objects in $\CX$. Then the functor $\Psi$, defined in Construction \ref{FirstCoonstr}, induces an equivalence of categories
$$ S_{\CX}(\CA)/ \mathcal{V} \simeq \mmod \underline{\CX}$$
\end{theorem}

For the definition of $S_{\CX}(\CA)$ and  proof of  the above theorem, we refer to the subsection  \ref{Firstconstuction}.\\
Among application of such generalization to prove one part of Theorem \ref{IndTheorem}, it has other interesting applications. For instance, as a immediate consequence of this generalization together with results  \cite[Theorem 3.6 ]{Cr} due to Simon Crawford, we can obtain the following  connections between
Kleinian singularities  and preprojective algebras, 
$$\Gprj T_2(R_Q)/ \mathcal{V} \simeq \mmod \Pi(\CQ)$$
where $R_{\CQ}$ is the coordinate ring of the corresponding Kleinian singularity for the Dynkin quiver $\CQ$, and $\Pi(\CQ)$ is the preprojective algebra of $\CQ.$
Also, as an another interesting application
 we consider a relative version of the following conjecture posted by Auslander:\\
{\bf Conjecture:} Suppose $\CA$ is an abelian category of enough projectives, $A$ is an object of $\CA,$ and $F$ is a direct summand of $\Ext^1(A, -)$. Then the question is whether $F$ is of the form $\Ext^1(B, -)$ for some $B$ in $\CA.$ \\
 In Theorem \ref{AusConj} and Proposition \ref{AUSGPrj} are given  a relative version of the above conjecture for the subcategories appeared in Theorem \ref{IndTheorem}.
  
In spite of our results for the  existence of Auslander-Reiten duality for subcategories, we will then try to compute the uniquely determined equivalence (up to isomorphism), $\tau_{\CX}$, for some certain cases, as follows. 
\begin{itemize}
\item[$\mbox{-}$] Let $\La$ be a $1$-Gorenstein algebra. Then $\tau_{\CP^{<\infty}(\La)}$ is computed in Theorem \ref{1-Gortrans} by the help of a result due to C.M Ringel and M. Schmidmeier  \cite[Proposition 2.4]{RS3}.  
\item[$\mbox{-}$] In Theorem \ref{CHArecterzationSemi}, we investigate when $\tau_{\mathcal{G}}$ is isomorphic to the first  syzygy functor. In addition, we collect in Example \ref{SgTran} some computations for $\tau_{\CG}$ has been done by others. 
\end{itemize}

Recently studding $1$-Gorenstein algebras has been attracted more attention in representation theory of artin algebras. In particular, the class of $1$-Gorenstein algebras attached to symmetrizable cartan matrices defined in \cite{GLS}. In such $1$-Gorenstein algebras, the  modules with finite projective dimension, also called locally free modules in there, play an important role.

This paper is organized as follows. In section 2, we recall some basic definitions and facts, and introduce the notations used throughout the paper. 
 The section 3 are divided into three subsection. In this section we will prove several   equivalences between  the quotient categories coming from morphism categories and the categories of finitely presented functors on the stable categories, including the  proof of Theorem \ref{Zeroequivalence}. Moreover, in there, a relative version  of the Auslander's conjecture will be given. Section 4 is devoted to give a proof for the first part of Theorem \ref{IndTheorem} and also to compute $\tau_{\CP^{<\infty}(\La)}$ when $\La$ is a $1$-Gorenstein algebra. In section 5  we will restrict our attention  to the subcategory $\Gprj \La$  because of  their importance in the homological and representation-theoretic structure of an artin algebra.  In this section, by applying our  functorial approach, we  characterize  when $\mmod (\underline{\rm{Gprj}} \mbox{-} \La)$ is a semisimple abelian category in terms of the Auslander-Reiten translation $\tau_{\mathcal{G}}$ in $\Gprj \La$. In particular we show it happens if  $\tau_{\mathcal{G}}$  is the first syzygy functor and providing many algebras having this property.
 
\begin{notation} If $R$ is  a ring with unity, then  $\Mod R$ denotes the category of all right $R$-modules and $\mmod R$ denotes its full
	subcategory consisting of all finitely presented modules. Moreover, $\Prj R$, resp. $\prj R$, denotes the full subcategory of $\Mod R$, resp. $\mmod R$, consisting of projective, resp. finitely generated projective, modules. Throughout the paper, $\Lambda$ denotes an artin algebra over a commutative artinian ring $k.$
\end{notation}

\section{Preliminaries}
We now recall some of the definitions and results that we will make use of throughout this
paper.
\subsection{Functor category}
Let $\CA$ be an  additive skeletally small category. The Hom sets will be denoted either by $\Hom_{\CA}( - , - )$, $\CA( - , - )$ or even just $( - , - )$, if there is no risk of ambiguity. Let $\CX$ be a full subcategory of $\CA$. In this paper we assume all  subcategories is full, closed under isomorphic of objects, i.e. if $X \in \CX$ and $X \simeq Y$ in $\CA$, then $Y \in \CX,$ and closed under finite direct sums. By definition, a (right) $\CX$-module is an additive functor $F:\CX^{\op} \rt \CA b$, where $\CA b$ denotes the category of abelian groups and $\CX^{\op}$ is the opposite category of $\CX$. The $\CX$-modules and natural transformations between them, forms an abelian category denoted by $\Mod\CX$ or sometimes $(\CX^{\op},\CA b)$. In the literature an object in $\rm{Mod}\mbox{-} \CX$ is also called a contravariant functor from $\CX$ to $\CA b $.  An $\CX$-module $F$ is called finitely presented if there exists an exact sequence $$\CX( - ,X) \rt \CX( - ,X') \rt F \rt 0,$$ with $X$ and $X'$ in $\CX$. All finitely presented $\CX$-modules form a full subcategory of $\Mod  \CX$, denoted by $\mmod \CX$ or sometimes $\rm{f.p.}(\CX^{op}, \CA b)$. A left $\CX$-module defined as a right $\CX^{\op}$-module, or a (covariant) functor from $\CX$ to $\CA b$. The category of left $\CX$-modules, or right $\CX^{\op}$-module,  and its  full subcategory consisting of finitely presented (left) $\CX$-modules, or  finitely presented right $\CX^{\op}$-module, will be denoted by $\CX$-Mod and $\CX$-${\rm mod}$, respectively. We sometimes use  notations $\Mod \CX^{\op}$ and  $\mmod \CX^{op}$, respectively to show them. Note that when we only  deal with $\mmod \CA$, there is no need to assume that $\CA$ is skeletally small category as there is no worry about set theoretic problem.

The Yoneda embedding $\CX \hookrightarrow \mmod\CX$, sending each $X \in \CX$ to $\CX( - ,X):=\CA( - ,X)\vert_{\CX}$,
is a fully faithful functor. Note that for each $X \in \CX$, $\CX( - ,X)$ is a projective object of $\mmod\CX$. Moreover, if $\CX$ has split idempotents, then every projective object of $\mmod\CX$ is  as  $\CX( - ,X)$, for some $X \in \CX$. 

A morphism $X \rt Y$ is a weak kernel of a morphism $Y \rt Z$ in $\CX$ if the induced sequence
$$( - , X)  \lrt ( - , Y)  \lrt ( - , Z)$$
is exact on $\CX$. It is known that $\mmod\CX$ is an abelian category if and only if $\CX$ admits
weak kernels, see e.g. \cite[Chapter III, \S 2]{Au2}.\\
Let $A \in \CA$. A morphism $\varphi : X \rt A$ with $X \in \CX$ is called a right $\CX$-approximation of $A$ if $\CA( - , X)\vert_{\CX} \lrt \CA( - ,A)\vert_{\CX} \lrt 0$ is exact, where
$\CA( - ,A) \vert_{\CX}$ is the functor $\CA( - ,A)$ restricted to $\CX$. Hence $A$ has a right $\CX$-approximation if and only if $( - ,A)\vert_{\CX}$ is a finitely generated objects of $\Mod\CX$. $\CX$ is called contravariantly finite if every object of $\CA$ admits a right $\CX$-approximation. Dually, one can define the notion of left $\CX$-approximations and covariantly finite subcategories. $\CX$ is called functorially finite, if it is both covariantly and contravariantly finite.

It is obvious that if $\CX$ is contravariantly finite, then it admits weak kernels and hence $\mmod\CX$ is an abelian category.

\subsection{Stable categories}\label{stableC}
Let $\CA$ be an abelian category with enough projective objects. Let $\CX$ be a subcategory
of $\CA$ containing $\Prj\CA$.
The stable category of $\CX$ denoted by $\underline{\CX}$ is a category whose objects are the same as those of
$\CX$, but the hom-set ${\underline{\CX}} (\underline{X}, \underline{X}')$ of $X, X' \in \CX$
is defined as $\underline{\CX}(\underline{X}, \underline{X}'):= \frac{{\CA}(X, X')}{\CP(X, X')}$, where $\CP(X, X')$
consists of all morphisms from $X$ to $X'$ that factor through a projective object. we usually use underline for objects to emphasize that we are  considering  them as an object in the  stable category. We have the canonical functor $\pi:\CX \rt \underline{\CX}$,
defined as identity on the objects but morphism $f:X\rt Y$ will be sent to the residue class $\underline{f}:=f+{\CP}(X, X').$\\

As defined above, let $\pi:\CX \rt \underline{\CX}$ be the canonical functor. It then induces an exact functor $\mathfrak{F}: \Mod\underline{\CX} \rt \Mod\CX$. It is not hard to see that $\mathfrak{F}$ in turn induces an equivalence between $\Mod\underline{\CX}$ and whose image $\rm{Im}\mathfrak{F}$ which is the full subcategory of $\Mod\CX$ consisting of those functors vanishing on $\Prj\CA$. We denote the image of $\mathfrak{F}$ in $\Mod \CX$ by $\rm{Mod}_0 \mbox{-} \CX$, and moreover  $\mmod_0 \CX=\rm{Im}\mathfrak{F}\cap \mmod \CX$.

\begin{lemma}\label{stabel=mmodd}
Let $\CA$ be an abelian category with enough projective objects and $\CX$ be a subcategory of $\CA$ containing $\Prj\CA$. Then we have the following commutative diagram
\[\xymatrix{\Mod\underline{\CX} \ar[r]^{\mathfrak{F} }& \rm{Mod}_0 \mbox{-}\CX \\ \mmod\underline{\CX} \ar@{^(->}[u] \ar[r]^{\mathfrak{F}{\vert}} &\mmodd\CX ,\ar@{^(->}[u] }\]
such that the  rows are equivalence and the others are inclusions. If furthermore $\CX$ is contravariantly finite, then $\mmod\underline{\CX}$ is an abelian category with enough projective objects.
\end{lemma}

\begin{proof}
See Proposition 4.1 of \cite{AHK}.
\end{proof}

The above result says us that one way of thinking  $\mmod \underline{\CX}$,  $\CX$ being contrarvariantly finite,  is considering it as a subcategory of $\mmod \CX$, that this  can be helpful as we will see later. One can also similarly  define  for  a subcategory $\CX$ consisting injective objects in an abelian category $\CA$ with enough injective objects, the stable category  $\overline{\CX}$ modulo injective modules.  In this case, since $\CA^{\op}$ holds the conditions of Lemma \ref{stabel=mmodd}  we  will also have a  dual version of  the Lemma. Denote by $\CX \mbox{-}_0 \rm{Mod}$, resp.  $\CX \mbox{-}_0 \rm{mod},$ the category of  covariant functors, resp.  finitely presented covariant functors, vanish on injective objects.



\subsection{Gorenstein projective modules}\label{gproj}

Recall that an $R$-module $G$ is called Gorenstein projective if it is a syzygy of a $\Hom_{R}( - ,\Prj R)$-exact exact complex
\[\cdots \rt P_{1} {\rt} P_0  {\rt} P^0 \rt P^1 \rt \cdots,\]
of projective modules. The class of Gorenstein projective modules is denoted by $\GPrj R$. Dually one can define the
class of Gorenstein injective modules $\GInj R$. We set $\Gprj R=\GPrj R \cap \mmod R$. One can see easily any projective module is Gorenstein projective. Moreover, if $\La$ is a self-injective algebra, then all modules are Gorenstein projective, in particular $\mmod \La=\Gprj\La$, and also if $\La$ has finite global dimension we don't have non-projective Gorenstein projective modules, i.e. $\Gprj \La= \rm{proj} \mbox{-} \La.$  $\La$ is
called virtually Gorenstein if $(\GPrj \La)^\perp = {}^\perp (\GInj \La)$, where orthogonal is taken with respect to $\Ext^1$, see \cite{BR}.
It is proved by Beligiannis and Henning \cite[Theorem 5]{BK} that if $\La$ is a virtually Gorenstein algebra, then $\Gprj \La$ is a contravariantly
finite subcategory of $\mmod \La$.
  Recall that $\La $ is called a $n$-Gorenstein Algebra, if injective dimensions of $\La$ as right and left module are less than or equal to $n$. It is known that Gorenstein algebras are Virtually Gorenstein but the converse cannot be true.
An algebra $\La$ is called of finite Cohen-Macaulay type, finite CM-type for short, if the number of indecomposable Gorenstein projective modules up to isomorphism is finite. It may be more reasonable to call such algebras of finite GP-type, but since "of finite CM-type" is more known terminology, we prefer to use this notation throughout the paper. 
Recall that the full subcategory $\rm{GPrj } \mbox{-} R$ of $\Mod R$ is closed under extensions. It
follows that $\rm{Gprj } \mbox{-} R$ is an exact category in the sense of Quillen. Moreover, it is a Frobenius
category, whose projective-injective objects are precisely projective $R$-modules.
Then by \cite [Theorem I.2.8]{H} the stable category  $\underline{\rm{GPrj }} \mbox{-} R$ modulo projective $R$-modules
has a natural triangulated structure: the translation functor is a quasi-inverse of the
syzygy functor, and triangles are induced by short exact sequences with terms in $\rm{GPrj } \mbox{-} R$. If $R$ is a right Noetherian, then we can consider $\underline{\rm{Gprj }}\mbox{-}R$ as a subtriangulated category of $\underline{\rm{GPrj }} \mbox{-} R$.

\subsection{Morphism category}
Let $\CA$ be an abelian category. The morphism category $\bf{H}(\CA)$ has as objects the
maps  in $\CA$, and morphisms are given by commutative diagrams. If we consider a map $f:A \rt B$ in $\mmod \La$
as an object of $H(\La),$ we will write either

$$
(A\xrightarrow{f} B) \quad \text{or} \quad \begin{array}{c}
A \\
\mathrel{\mathop{\downarrow}^f}
\\
B
\end{array} $$
but often also just $f$ or without parenthesis, whenever will be convenient and not misleading.
This category is again an abelian category. If $\CA$ is $\Mod R$, we will show by $\mathbf{H}(R)$ instead of $\mathbf{H}(\Mod R)$. In fact, $\mathbf{H}(R)$  is equivalent to $\Mod T_2(R)$, where
$T_2(R)=\bigl(\begin{smallmatrix}
R & R \\
0 & R
\end{smallmatrix}\bigr)$,  the ring of upper triangular $(2\times 2)$-matrices with coefficients in $R$.

One can also consider $\mathbf{H}(R)$ as representations over quiver $\mathbb{A}_2:\bullet \rightarrow \bullet$ by $R$-modules,   denoted by $\rm{Rep}(\mathbb{A}_2, R)$. In the rest of paper we shall use these identifications completely free, so there is no difference for us between morphisms in $\mmod \La$ or representations of $\mathbb{A}_2$
and modules over algebra $T_2(R).$ The subcategory of morphism $f:A \rt B$ in which $A$ and $B$ are finitely generated $R$- modules presented by $\rm{H}(R)$, or $\rm{rep}(\mathbb{A}_2, R)$,  that by the equivalence mapped to the  finitely generated $T_2(R)$-modules.\\
The study of Gorenstein projective presentations of a (not necessary finite) quiver over arbitrary rings were studied for instance  in \cite{EHS} and \cite{EEG}.  The aim of these papers was to give a local characterization of Gorenstein projective modules. The case of our interest here is only when   quiver is $\mathbb{A}_2$, then we have:
\begin{lemma}\label{GPCH} Assume $R$ is an arbitrary ring. Then a representation $A \st{f} \rt B$ of $\mathbb{A}_2$ is Gorenstein projective if and only if
\begin{itemize}
\item[$(1)$] $A, B$ and $\rm{Coker}(f)$ are Gorenstein projective;
\item[$(2)$]  $f$ is a monomorphism.
\end{itemize}
\end{lemma}
\begin{proof}
See \cite[Thoreom 3.5.1]{EHS}
\end{proof}
This characterization can be helpful to transfer some Gorenstein property from $R$ to $T_2(R)$, e.g. if  $\La$ is a virtually Gorenstein algebra then   $T_2(\La)$ so  is, see \cite[Thoreom 3.4.3]{EHS}. But note that being of finite  $\rm{CM}$-type cannot be preserved, if $\La=K[x]/(x^n)$ and $K$ be an algebraic closed field, then $T_2(K[x]/(x^n))$ isn't of finite $\rm{CM}$-type for $n>5.$ But in Section 5, we will discuss kinds of algebras which preserve this property.

\subsection{Quotient category}\label{objective}
We define the kernel of an additive functor $F :\mathcal{A} \rt \mathcal{B}$, between additive categories $\CA$ and $\CB$, as the full subcategory of all objects $A$ in $\mathcal{A}$ such that $F(A)=0$.  For   a given full subcategory $\mathcal{C}$ of $\mathcal{A}$,  the ideal $\rm{I}_{\mathcal{C}}$ of $\mathcal{A}$, or the  sub functor $\rm{I}_{\mathcal{C}}$ of the bifunctor $\Hom(-,-)$,  for any $X, Y \in \rm{Ob}(\mathcal{A})$  defined as subgroup  $I_{\mathcal{C}}(X, Y )$ of $\Hom_{\mathcal{A}}(X, Y ) $ consisting those morphisms  that factor through an object in $\mathcal{C}$.
By a quotient category we mean the additive quotient. That is, given
an additive category $\mathcal{A}$ and a ideal $\mathcal{I}$ of $\mathcal{A}$, the quotient category $\mathcal{A}/\mathcal{I}$ has the same objects as  $ \rm{Ob}(\mathcal{A})$ and for any  $X, Y \in \rm{Ob}(\mathcal{A})$
$$\Hom_{\mathcal{A}/\mathcal{C}}(X, Y ) := \frac{\Hom_{\mathcal{A}}(X, Y )} {I_{\mathcal{C}}(X, Y )}.$$
\subsection{ِِDualizing varieties}\label{DualityVaraity1}
Let $k$ be a commutative artinian ring.  Let $\CX$ be an additive $k$-linear essentially small category. $\CX$ is called a dualizing $R$-variety if the functor $\Mod\CX \lrt \Mod\CX^{\op}$ taking $F$ to $DF$, induces a duality $\mmod\CX \lrt \mmod\CX^{\op}$. Note that $D( - ):=\Hom_k( - ,E)$, where $E$ is the injective envelope of $k/{\rm rad}k$.
 Every functorially finite subcategory of $\mmod \La$ is  a dualizing $k$-variety \cite[Theorem 2.3]{AS}.

\section{From morphism categories to the finitely presented functors over the stable categories}

In this section we provide several equivalences between quotient categories of certain   subcategories  of morphism category and the categories of
finitely presented functors over the stable categories.
\subsection{The First  equivalence} \label{Firstconstuction}

For a given subcategory $\CX$ of an abelian category  $\CA$ we assign the subcategory $S_{\CX}(\CA)$ of $\mathbf{H}(\CA)$ consisting morphism $A \st{f} \rt B$ satisfying:
\begin{itemize}
\item[$(i)$] $f$ is a monomorphism;
\item[$(ii)$] $A$, $B$ and $\rm{Coker} (f)$  belong to $\CX.$
\end{itemize}
\begin{example}
By lemma \ref{GPCH} if $R$ is a right Notherian ring, then $S_{\Gprj R}(\mmod R)$, or simply $S(\Gprj R)$,  is equivalent to the subcategory  of finitely generated Gorenstein projective modules over $T_2(R).$  Let $K$ be a field and setting $\Lambda_n=K[x]/(x^n)$,  as $\La_n$ is self-injective, then $\Gprj \La=\mmod \La$. The subcategory $S(\mmod \Lambda_n)$, or $S(n)$ as used in \cite{RZ}, is the subcategory of invariant subspace of nilpotent operators with nilpotency index at most $ n$. The structure of this subcategory 
has been studied more recently by Ringel and Schmidmeier in \cite{RS1} and \cite{RS2}. Moreover, take $\rm{Flat} \mbox{-} R$, the subcategory of right flat modules, then $S_{\rm{Flat} \mbox{-} R}(\Mod R)$  will be the subcategory of  flat modules over $T_2(R)$ by \cite{EOT}.
\end{example}
 In this subsection we have plan to show that $\mmod \underline{\CX}$ can be described as a quotient category of $S_{\CX}(\CA).$ To this end, we define functor $\Psi: S_{\CX}(\CA) \rt \mmod \underline{\CX}$ respect to the  subcategory $\CX$ as follows.
 \begin{construction}\label{FirstCoonstr}

Taking an object $A \st{f} \rt B$ of $S_{\CX}(\CA)$, then we have the following short exact sequence
$$0 \rt A \st{f} \rt  B \rt \rm{Coker}(f) \rt 0$$ in $\CA, $ this in turn gives the following short exact sequence
$$ (*) \ \ \ 0 \lrt (-, A) \st{(-, f )} \rt (-, B) \rt (-, \rm{Coker}(f)) \rt F \rt 0$$
in $\mmod \CX$. In fact, $(*)$  corresponds to a projective resolution of $F$ in $\mmod \CX.$ We define $\Psi(A \st{f} \rt B)=: F$.

For morphism: Let $\sigma=(\sigma_1, \sigma_2)$ be a morphism from $A \st{f} \rt B$ to $A' \st{f'} \rt B'$, we can make the following diagram
\[\xymatrix{0 \ar[r] & A \ar[r]^{f} \ar[d]^{\sigma_1} & B \ar[r] \ar[d]^{\sigma_2} & \rm{Coker}(f)  \ar[r] \ar[d]^{\sigma_3} & 0 \\
0 \ar[r] & A'  \ar[r]^{f'} & B' \ar[r] & \rm{Coker}(f')  \ar[r]  & 0,}\]
that $\sigma_3$ can be determined uniquely by $\sigma_1$ and $\sigma_2.$ By appying the Yoneda lemma, the above diagram gives  the following diagram

\[\xymatrix{0 \ar[r] & (-, A) \ar[d]_{(-, \sigma_1)} \ar[r]^{(-, f)} & (-, B) \ar[d]^{(-, \sigma_2)} \ar[r] & (-, \rm{Coker}  (f)) \ar[d]^{(-, \sigma_3)} \ar[r] & F \ar[d]^{\overline{(-, \sigma_3)}} \ar[r] & 0 \\ 0  \ar[r] & (-, A')\ar[r]^{(-, f')} & (-, B')\ar[r] & (-,\rm{Coker} (f') )\ar[r] & F' \ar[r] & 0.}\]
in $\mmod \CX.$ Define $\Psi(\sigma):= \overline{(-,\sigma_3)}$, which is obtained uniquely by $\sigma_1$ and $\sigma_2$.

\end{construction}

\begin{theorem} \label{Thefirst}
Let $\CA$ be an abelian category with enough projectives.
Let $\CX$ be a subcategory of $\CA$ including $\rm{Proj} \mbox{-}\CA$, being contravariantly finite and closed under kernel of epimorphisms. Consider the full subcategory $\mathcal{V}$ of $S_{\CX}(\CA)$ formed by  finite direct sums of   objects in the form of $(X \st{\rm{Id}}\rt X )$ or $ (0 \rt X )$, that $X$ runs through  all of objects in $\CX$. Then the functor $\Psi$, defined in above construction, induces the following equivalence of categories
$$ S_{\CX}(\CA)/ \mathcal{V} \simeq \mmod \underline{\CX}.$$
\end{theorem}
\begin{proof}
The functor $\Psi$ is dense. It can be seen by Lemma \ref{stabel=mmodd}  that any functor in $\mmod \underline{\CX}$ obtained by a short exact sequence in $\CA$ (with its two right terms belong to $\CX$), and then the monomorphism in the short exact sequence gives us an object in $S_{\CX}(\CA)$, by applying the assumption $\CX$ being closed under kernel of epimorphisms. Any morphism between functors in $\mmod \underline{\CX}$ is  considered  as a morphism in $\mmod \CX$, so  lifted to the  corresponding projective resolution. Then by using Yoneda's Lemma we can obtain a morphism in $S_{\CX}(\CA)$ to prove fullness. If $\Psi(A \st{f} \rt B)=0$, then by definition,
$$0 \rt (-, A) \st{(-, f)} \lrt (-, B)  \st{(-, g)} \lrt (-, \rm{Coker}(f)) \rt 0.$$
Now by evaluation the above short exact sequence in  $\rm{Coker}(f)$, as it belongs to $\CX,$ we get
$$0 \rt A \st{f} \rt B \rt \rm{Coker}(f) \rt 0$$
is a split short exact sequence. This gives us $A \st{f} \rt B$ being isomorphic to the  object $(A \rt \rm{Im}(f)) \oplus (\rm{Ker}(g) \rt 0)$ in $\mathbf{H}(\CA)$, and so $f$ belongs to $\mathcal{V}$. Note that $\rm{Ker}(g) \in \CX$: as $A\simeq \rm{Im}(f)$ then $\rm{Im}(f)$ is in $\CX$,  and on the other hand by the following split short exact sequence
$$0 \rt \rm{Ker}(g) \rt B\simeq \rm{Im}(f)\oplus \rm{Ker}(g) \rt \rm{Im}(f) \rt 0,$$
and using our assumption being closed under kernel of epimorphisms proves our claim. Assume that $\Psi(\sigma)=0$, for $\sigma=(\sigma_1, \sigma_2): (A \st{f} \rt B) \rt (A' \st{f'} \rt B')$ in $S_{\CX}(\La)$. Therefore, we have

\[\xymatrix{0 \ar[r] & (-, A) \ar[d]_{(-, \sigma_1)} \ar[r]^{(-, f)} & (-, B) \ar[d]^{(-, \sigma_2)} \ar[r] & (-, \rm{Coker}(f)) \ar[d]^{(-, \sigma_3)} \ar[r] & F \ar[d]^{0} \ar[r] & 0 \\ 0  \ar[r] & (-, A')\ar[r]^{(-, f')} & (-, B')\ar[r] & (-,\rm{Coker}(f') )\ar[r] & F' \ar[r] & 0.}\]

Since the first (resp. second) row of the above diagram is a projective resolution for $F$ (resp. $F'$) in $\mmod \CX$. Then considering the following diagram
\[\xymatrix{\cdots \ar[r] & (-, A) \ar[d]_{(-, \sigma_1)} \ar[r]^{(-, f)} & (-, B) \ar[d]^{(-, \sigma_2)} \ar[r] & (-, \rm{Coker}(f)) \ar[d]^{(-, \sigma_3)} \ar[r] &  \cdots \\ \cdots  \ar[r] & (-, A')\ar[r]^{(-, f')} & (-, B')\ar[r] & (-,\rm{Coker}(f') )\ar[r] &  \cdots,}\]
as a chain map in $\mathbb{C}^{\rm{b}}(\mmod \CX)$, the category of bonded complexes on $\mmod \CX,$ should be null-homotopic. Hence by using a standard argument, the above chain map should  be factored through projective complexes in $\mathbb{C}^{\rm{b}}(\mmod \CX)$ as follows:

$$\xymatrix{    & 0 \ar[r] & (-, A)   \ar@/^1.25pc/@{.>}[dd] \ar[d]  \ar[r]  & (-, B)  \ar@/^1pc/@{.>}[dd] \ar[d]  \ar[r]  &  (-, \rm{Coker}(f))  \ar@/^1.2pc/@{.>}[dd] \ar[d]  \ar[r]  & 0 \ar@/^1.2pc/@{.>}[dd] \ar[d] &   \\
	&0 \ar[r] & (-, A') \ar[r] \ar[d] & (-, A'\oplus B') \ar[r] \ar[d]  &   (-, B'\oplus\rm{Coker}(f') )  \ar[d]  \ar[r]  & \rm{Coker}(f )\ar[d]  \\
	& 0 \ar[r] & (-, A')  \ar[r]  & (-, B')  \ar[r]  &(- \rm{Coker}(f')) \ar[r] & 0. }$$

 This factorisation as the above  gives us a factorization of morphism $\sigma$ through the direct sum of $A' \st{\rm{Id}} \rt A'$  and $0 \rt B'$ in  $S_{\CX}(\CA).$ So $\Psi$ is a objective functor in the sense of \cite{RZ}.

 So $\Psi $ is  full, dense and objective. Hence by \cite[Appendix]{RZ}, $\Psi $ induces an equivalence between $S_{\CX}(\CA)/ \mathcal{V}$ and $\mmod \underline{\CX}$. So we are done.

\end{proof}

 The two  categories in the equivalence of the above theorem in a natural way can be dualized. Hence we have a dual version of the theorem. To sate the dual, we need some notations. For subcategory $\CX$ of $\CA$ let $F_{\CX}(\CA)$ be the  subcategory of $\mathbf{H}(\CA)$ containing morphism $A \st{f}\rt B$ satisfying:
\begin{itemize}
\item[$(1)$] $f$ is an epimorphism;
\item[$(2)$] $A, B$ and $\Ker (f)$ are in $\CX.$
\end{itemize}
Functor $\Psi':F_{\CX}(\CA) \rt (\overline{\CX}\mbox{-} \rm{mod})^{\rm{op}}$ respect to the subcategory $\CX$ can be defined similar to $\Psi,$ mapping  object $(A \st{f} \rt B)$ into  a functor fitting in the following exact sequence
$$0 \rt (B, -) \st {(f, -)} \rt (A, -) \rt (\Ker(f), -) \rt F \rt 0 .$$
and morphisms can be defined in the obvious way.
\begin{theorem}\label{Thesecound}
Let $\CA$ be an abelian category with enough injectives.
Let $\CX$ be a subcategory of $\CA$ including $\rm{Inj} \mbox{-}\CA$, being  covariantly finite and closed under cokernel of monomorphisms. Consider the full subcategory $\mathcal{V'}$ formed by finite  direct sums of objects of the forms $(X \st{\rm{Id}}\rt X )$ or $ (X \rt 0 )$,  that $X$ runs through  all of objects in $\CX$. Then the functor $\Psi'$, defined above, induces the following equivalence of categories
$$ F_{\CX}(\CA)/ \mathcal{V'} \simeq (\overline{\CX} \mbox{-} \rm{mod})^{\rm{op}}.$$
\end{theorem}
\begin{proof}
The proof is dual of Theorem \ref{Thefirst}, or deducing directly  from it by setting  $\CA^{\op}$ as an abelian category with enough projectives  and $\CX^{\op}$ as a subcategory of $\CA^{\op}$ satisfying the assumptions of the theorem .
\end{proof}
We use the same notation for the  induced  equivalences in Theorem \ref{Thefirst} and Theorem \ref{Thesecound}.
Sometimes we consider $\Psi'$ as a duality from $F_{\CX}(R)/ \mathcal{V'}$ to $\overline{\CX} \mbox{-} \rm{mod}.$ 

For a subcategory $\CX $ of $\CA$, we attached two subcategories $F_{\CX}(\CA)$ and $S_{\CX}(\CA)$ of $\mathbf{H}(\CA)$. These two subcategories are 
related by the kernel and cokernel functors:
$$\rm{Cok}: S_{\CX}(\CA) \rightarrow F_{\CX}(\CA), \qquad (A\xrightarrow{f} B) \mapsto (B\xrightarrow{\text{can}} \Coker(f)) $$
$$\Ker: F_{\CX}(\CA) \rightarrow S_{\CX}(\CA), \quad\qquad (A\xrightarrow{g} B) \mapsto (\Ker(g)\xrightarrow{\text{inc}} A) .$$

The kernel and cokernel functors are in fact a pair of of inverse equivalences. As $\mathcal{V}$ and $\mathcal{V'}$ are preserved by $\rm{Cok}$ and $\Ker$, respectively. Then we have a pair of inverse equivalences  for the corresponding quotient categories.  we will use the same symbol here for the induced pair of inverse equivalences between $F_{\CX}(\CA)/\mathcal{V}'$ and $S_{\CX}(\CA)/\mathcal{V}$.

Let $\CX$ be a subcategory of an abelian category with enough projectives and injectives. Assume $\CX$ contains all projective objects and injective ones, being functorially finite, being closed under kernel and cokernel of epimorphisms and monomorphisms, respectively.
 Let $\Phi$ and $\Phi'$ be the following compositions, respectively:
$$\mmod \underline{\CX} \st{\Psi^{-1}} \rt S_{\CX}/\mathcal{V} \st{\rm{Cok}} \rt F_{\CX}/\mathcal{V}' \st{\Psi'}  \rt \overline{\CX} \mbox{-} \rm{mod}$$

$$\mmod \underline{\CX} \st{\Psi} \leftarrow S_{\CX}/\mathcal{V} \st{\rm{Ker}} \leftarrow F_{\CX}/\mathcal{V}' \st{(\Psi')^{-1}}  \longleftarrow \overline{\CX} \mbox{-} \rm{mod}.$$
where $\Psi^{-1}$, resp. $(\Psi')^{-1}$, is the quasi-inverse of the functor $\Psi$, resp. $\Psi'$, defined in Theorem \ref{Thefirst} and Theorem \ref{Thesecound}.
Then $\Phi$ and $\Phi'$ are a pair of inverse equivalences  respect to the subcategory $\CX$.
\begin{proposition}\label{enoughinj} Let $\CX$ be a functorially finite  subcategory of an abelian category $ \CA$ that has enough projectives and injectives. Assume $\CX$ includes all projective objects  and injective objects in $\CA$ and closed under kernel of epimorphisms and cokernel of monomorphisms. Then $\mmod \underline{\CX}$ and $\overline{\CX} \mbox{-} \rm{mod}$ are both abelian categories with  enough injectives.
\end{proposition}
\begin{proof}
We know that a duality is an exact functor and mapping projective objects to injectivs. On the other hand, $\mmod \underline{\CX}$ and $\overline{\CX} \mbox{-} \rm{mod}$ have always enough projectives. Then  the dualities $\Phi$ and $\Phi'$, defined above, follow the statement. Note that by Lemma \ref{stabel=mmodd}, $\mmod \underline{\CX}$ and $\overline{\CX} \mbox{-} \rm{mod}$ are both abelian categories.
\end{proof}

Recall that a subcategory $\CX$ of $\CA$ is resolving if it contains all projectives, and if for every exact
sequence $0 \rt A \rt B \rt C \rt 0 $ in $\CA$ we have $A, C \in \CX$ implies $B \in X,$ and
$B,C \in \CX$ implies $A \in  \CX.$ Coresolving subcategories are defined dually.
\begin{example}\label{Examplefirst}
\begin{itemize}
\item [$(i)$] It is obvious whenever $\CX$ is the entire of abelian category $\CA$, then it satisfies all conditions of  Proposition \ref{enoughinj}. A non-trivial example of $\CX$ is the subcategory $\mathcal{P}^{< \infty}(\La)$ of finitely generated modules of finite projective dimension (or finite injective dimension)  on  Gorenstein algebra $\La$. It is proved
over a Gorenstein ring a module has finite projective dimension if and only if it has finite injective dimension. This fact implies that $\mathcal{P}^{< \infty}(\La)$ contains all injective modules. It is also easy to check that just by using the long exact sequence that $\mathcal{P}^{< \infty}(\La)$ is closed under kernel and cokernel of epimorphisms and monomorphisms, respectively. It remains to prove that $\mathcal{P}^{< \infty}(\La)$ is functorially finite. On Gorenstein algebras we have cotorsion pair $(\rm{Gprj} \mbox{-} \La, \mathcal{P}^{< \infty} )$ in $\mmod \La$ which is also complete, hence $\mathcal{P}^{< \infty}(\La)$ is a covariantly finite subcategory in $\mmod \La$, see e.g. \cite{EJ} for the facts concerning Gorenstein ring. Since $\mathcal{P}^{ < \infty}(\La)$ is  coresolving and covariantly finite, then so is contravariantly finite, by \cite{KS}. Hence $\mathcal{P}^{ < \infty}(\La)$ is a  functorially finite subcategory of $\mmod \La$. Let $H=H(C, D, \Omega)$ be the attached 1-Gorenstein algebra to symmetrizable generalized Cartan matrix $C$ and an orientaion $\Omega$ of $C$ and symetraizer $D$ of $C$, see \cite{GLS}. Let $\rm{rep}_{\rm{l.f.}}(H)$ be the subcategory of all locally free $H$-modules. By \cite[Theorem 1.2]{GLS} we get $\mathcal{P}^{< \infty}(H)= \rm{rep}_{\rm{l.f.}}(H)$,  so holding the conditions of \ref{enoughinj}.
\item[$(ii)$] Let $\La$ be a self injective algebra then  the stable category $\underline{\rm{mod}}\mbox{-} \La$ is a triangulated category such that whose triangles coming from short exact sequences in $\mmod \La$, see \ref{gproj}. Let $\pi: \mmod \La \rt \underline{\rm{mod}} \mbox{-} \La $ be the canonical functor and $\CY$ a sub triangulated in $\underline{\rm{mod}}\mbox{-} \La$. Then the preimage $\pi^{-1}(\CY)$ is a subcategory of $\mmod \La$ that contains projective-injective modules in $\mmod \La$, and closed under kernel, resp. cokernel,  of epimorphisms, resp. monomorphisms, that it can be shown just by using the structure of triangles in $\underline{\rm{mod}} \mbox{-} \La$. If we also assume $\La$ is of finite representation type, then $\pi^{-1}(\CY)$  becomes a  functorial finite  subcategory in $\mmod \La$, and  hence $\pi^{-1}(\CY)$ satisfies  the conditions of Proposition \ref{enoughinj}.
\item[$(iii)$] Assume that $(^{\perp} \xi, \xi, \xi^{\perp})$ is a cotorsion triple in an abelian category $\CA$ with enough projectives and injectives, that is, $(^{\perp} \xi, \xi)$ and $(\xi, \xi^{\perp})$ are   cotorsion theories in $\CA,$ see \cite[Chapter VI]{BR}. Then, by definitions, $\xi$ satisfies the conditions we need. Let us explain a way to produce new  cotorsion triple by the given one.  Let $\CQ $ be an acyclic quiver with finite  number of vertexes and set $\rm{Rep}(\CQ, R)$ the category of representations of $\CQ$ by $R$-modules and $R$-homorphismes. Denote $\CQ(\xi)$ the class of all representations $\CX$ in $\rm{Rep}(\CQ, R)$ whose vertexes are represented by an object in $\xi$. Then by using \cite[Theorem A]{EHHS} $\CQ(\xi)$ is a middle term of a triple cotorsion theory in $\rm{Rep}(\CQ, R)$. Let us point out although \cite[Theorem A]{EHHS} stated for module category of a ring but the proof is valid for any abelian category with enough projectives and injectives. For example, $(\Mod R, \rm{Inj}\mbox{-}R= \rm{Prj}\mbox{-}R, \Mod R)$ is a  cotorsion triple theory in $\Mod R$ whence $R$ is a quasi-Frobenius ring, so by this trivial  cotorsion triple we can construct  new subcategory  $\CQ(\rm{Inj}\mbox{-}R)$ in $\rm{Rep}(\CQ, R)$  satisfying our conditions.  Note that example $(i)$ is a special case of this example since $(\rm{Gprj}\mbox{-} \La, \mathcal{P}^{<\infty}(\La), \rm{Ginj}\mbox{-}\La )$ is a  cotorsion triple in $\mmod \La$ as $\La$ is a Gorenstein algebra. 
\end{itemize}
\end{example}
\begin{lemma}\label{extension}
Let $\CX$ be a subcategory in an abelian category $\CA$ with enough projective objects and injective objects satisfying the assumptions of Proposition \ref{enoughinj}. Then $\CX$ is closed under extension, i.e. if $0 \rt  X \rt  Y \rt Z \rt 0$ is exact with $X$ and $Z$ in $\CX,$ then $Y$ is in $\mathcal{X}.$
\end{lemma}
\begin{proof}

Consider a short exact sequence
$$ 0 \rt X \rt Z \rt Y \rt 0$$
such that $X$ and $Y$ are in $\CX.$  Since $\CA$ has enough injectives then so there is a monomorphisms $X \rt I$ for some injective object $I$ in $\CA$. Now we then take the pushout to get the commutative diagram below.
$$\xymatrix{ & 0 \ar[d] & 0 \ar[d] &  &\\
0 \ar[r] & X \ar[r] \ar[d] & Z  \ar[r] \ar[d] & Y  \ar@{=}[d]\ar[r] &  0\\
0 \ar[r] & I \ar[r] \ar[d] & U  \ar[r] \ar[d] & Y \ar[r] & 0\\
 & \Omega^{-1}(X) \ar@{=}[r] \ar[d] & \Omega^{-1}(X) \ar[d] & & \\
& 0  & 0 & & }$$
Recall as we fixed for subcategories in this paper to be closed under finite sums, we can deduce that $U\simeq I\oplus Y$ belongs to $\CX$. Since $\CX$ is closed under kernel of epimorphisms, the short exact sequence in the middle culmen in the above diagram implies that $Z \in \CX$. So we are done.

\end{proof}
As an immediate application of the above Lemma we have:
\begin{corollary}
Let $\CX$ be a subcategory in an abelian category $\CA$ with enough projective objects and injective objects. Then $\CX$ holding conditions \ref{enoughinj} if and only if it is functorialy finite, resolving and corsolving subcategory of $\CA.$
\end{corollary}

\begin{remark}
We refer to \cite[Chapter VI]{BR} for  more properties of this sort of subcategories satisfying the above equivalent conditions. The intersting point in Theorem 3.2 of \cite[Chapter VI]{BR} is if $\CX$ satisfying the condition, i.e, functorially finite resolving and coresolving subcategory of a Krull-Schmidt $\CA$ with enough projective and injective objects,  then the converse  is true, meaning there exist a cotorsion triple in $\CA$ such that $\CX$ is the middle term of it. In fact, in this case, there exists a bijection between cotorsion triple and this sort of subcategories, see Corollary 4.10 of \cite[Chapter VI]{BR}.
\end{remark}

\subsection{ A relative version of   Auslander's conjecture}\label{AusCOnjsec} In this subsection,  we  investigate Auslander's conjecture sated in the Introduction. We show that a direct summand of $\rm{Ext}$-functor over an abelian category sufficiently nice is again of that form, in addition, a relative version of the conjecture will also  be given.

 Let $\CX$ be a subcategory of an abelian category with enough projectives and injectives. The   functor $\Ext_{\CA}^1(-,X)\vert_{\CX}$, (resp. $\Ext_{\CA}^1(X, -)\vert_{\CX}$) is only the restriction of the usual $\Ext$-functor  on $\CX$, we will delete $\vert_{\CX}$ whenever $\CX$ is the entire of abelian category $\CA$. Recall that as we fixed before to simplify we use $(-, \underline{X})$ for $\underline{\CX}(-, X),$ (resp.  $(\overline{X}, -)$ for $\overline{\CX}(X, -)$,) for each $X$ in $\CX.$

\begin{theorem}\label{AusConj} Let $\CX$  be a subcategory of an abelian category $ \CA$ with enough projectives and injectives,  and  $\CX$  holding  the conditions of proposition \ref{enoughinj}. Then
\begin{itemize}
\item[$(i)$]
If $\underline{\CX}$ has split idempotents, then
\begin{itemize}
\item[$(1)$] For every object  $A$ in $\CX$, if $F$ is a direct summand of $\Ext_{\CA}^1 (A, -)\vert_{\CX}$, then there exists $B \in \CX$ such that $F\simeq \Ext_{\CA}^1 (B,-)\vert_{\CX}$
\item[$(2)$]In particular, for Krull-schmit  category $\CX=\CA$ the Auslander's conjecture is true.
\end{itemize}
\item[$(ii)$]
If $\overline{\CX}$ has split idempotents, then
\begin{itemize}
\item[$(1)$] For every object  $A$ in $\CX$, if $F$ is a direct summand of $\Ext_{\CA}^1 (-, A)\vert_{\CX}$, then there exists $B \in \CX$ such that $F\simeq \Ext_{\CA}^1 (-,B)\vert_{\CX}$
\item[$(2)$]In particular, for Krull-Schmit  category $\CX=\CA$ the dual of Auslander's conjecture is true.
\end{itemize}
\end{itemize}
\end{theorem}
\begin{proof}
By definition of $\Phi$, it  can be seen that for  $X \in \CX$, $\Phi(-, \underline{X})= \Ext_{\CA}^1(X, -)\vert_{\CX}$, (resp. $\Phi'( \overline{X}, -)=\Ext_{\CA}^1(-, X)\vert_{\CX}$ ). Let us check one of them. We have the following  short exact sequence
\begin{equation*}
0 \rt (-, \Omega(X)) \rt (-, P) \rt (-, X) \rt (-, \underline{X}) \rt 0
\end{equation*}
 in the category $\mmod \CX$ and
\begin{equation*}
0 \rt (X, -) \rt (P, -) \rt (\Omega(X), -) \rt \Ext_{\CA}^1(X, -)\vert_{\CX} \rt 0
\end{equation*}
in the category $\CX \mbox{-} \rm{mod}.$ Then these two short exact sequences help us to calculate $\Phi((-, \underline{X}))$ as follows:
 \begin{align*}
       \Phi((-, \underline{X}))&=\Psi' \circ \rm{Cok} \circ \Psi^{-1} ((-, \underline{X})) \\
       &= \Psi' \circ \rm{Cok}(\Omega(X)\rt P) \\
       &=\Psi' (P \rt X)\\
       & \simeq \Ext_{\CA}^1(X, -)\vert_{\CX}.
    \end{align*}
    
It is known that the projective objects of $\mmod \mathcal{C}$ are precisely the representable functors for a preadditive category $\mathcal{C}$ that has split idempotents and has finite direct sums. Now the statements can be followed by  dualities $\Phi$ and $\Phi',$ have already constructed.

\end{proof}
We remark that if $\CA$ is a Krull-schmit category, then the conjecture is valid. Thus Theorem \ref{AusConj} generalize \cite{M} to a Krull-Schmit abelian category. See also Proposition \ref{AUSGPrj} for another  subcategories (not necessary satisfying assumption of \ref{enoughinj}) such that a relative version of this conjecture holds.

As an immediate consequence of the proof of Theorem \ref{AusConj} is:
\begin{corollary}Let $\CX$  be a subcategory  of abelian category $ \CA$ and assume $\CX$ and $\CA$ hold the conditions of proposition \ref{enoughinj}. Then
\begin{itemize}
\item[$(i)$]
If $\underline{\CX}$ has split idempotents, then any injective functor  in $\overline{\CX} \mbox{-} \rm{mod}$ is of the form $\Ext_{\CA}^1(-, B)\vert_{\CX}$ for some $B$ in $\CX$.
\item[$(ii)$]
If $\overline{\CX}$ has split idempotents, then any injective functor  in $\mmod \underline{\CX} $ is of the form $\Ext_{\CA}^1(-, B)\vert_{\CX}$ for some $B$ in $\CX$
\end{itemize}

\end{corollary}
Theorem \ref{AusConj} says the validity of Auslander's conjecture  for an abelian category with enough projective and injective is equivalent to check that when its stable category has split idempotents. So by combining our results together with \cite{Au2}  we have:
\begin{proposition}
Let $\CA$ be an abelian category with enough projectives and injectives. Suppose either of
the following situations is satisfied:
\begin{itemize}
\item [(i)] $\CA$ has countable coproducts.
\item [(ii)] $\rm{gldim} \mbox{-} \CA$, supermom of projective dimension of all abjects, is finite.
\end{itemize}
Then $\underline{\CA}$ has split idempotents.
\end{proposition}

\subsection{The second equivalence}
The aim of  this section is to present another connection between categories of morphisms and (covariant) functors. Since, this turn,  finitely presented injective functors play a significant role, we need to explain some known results in concern of them.\\
 \begin{notation}
  For a left module $M$, let $-\otimes _{\Lambda}M:\mmod \Lambda \rt \CA b$ be the covariant functor, simply, by sending
  right $\La$-module $N$ to $N \otimes_{\La} M$. For a subcategory $\CX$ of $\mmod \La^{\op}$  denote by $-\otimes _{\Lambda}M \vert_{\CX}$, the restriction of $-\otimes _{\Lambda}M$ on the  subcategory $\CX.$   Also, in a similar way, we can define notation  $\Hom_{\La}(-, M)\vert_{\CX}$ and  $\Hom_{\La}(M, -)\vert _{\CX}$. For abbreviation, we usually  use $(-, M)\vert_{\CX}$
  $(M, -)\vert_{\CX}$, and moreover, as we did before,  $\vert_{\CX}$ will be deleted  whenever $M \in \CX.$
 \end{notation}

Let $R$ be an arbitrary ring. It is proved by Auslander \cite[Lemma 6.1]{Au1} that for a left $R$-module $M$, the covariant functor $ - \otimes_R M$ is finitely presented if and only if $M$ is a finitely presented left $R$-module. It is known that there is a full and faithful functor $T: R{\mbox{-}\rm{mod}} \lrt (\mmod R)\mbox{-}{\rm mod}$, defined by the attachment $M \mapsto {( - \otimes_R M)}$. Gruson and Jensen \cite[5.5]{GJ} showed that the category $(\mmod R)\mbox{-}{\rm mod}$ has enough injective objects and injectives are exactly those functors isomorphic to a functor of the form $ - \otimes_RM$, for some left $R$-module $M$, see also \cite[Proposition 2.27]{Pr}.

It is natural to ask what we can say for $\rm{Inj}(\CX{\mbox{-}}{\rm mod})$, when $\CX$ is a subcategory of module category. In the case that ring $R$ is an artin algebra $\Lambda$ we have the following result for  $\rm{Inj}(\CX{\mbox{-}}{\rm mod})$.

Let $\CX$  be a subcategory of $\mmod\La$ which is a dualizing $k$-variety.  Note that $\mmod \Lambda$ itself  is  a dualizing $k$-variety. Consider  the following commutative diagram
\[\xymatrix{\mmod(\mmod\La) \ar[rr]^{D} \ar[d]^{{\vert}_{{}_{\CX}}} && (\mmod \La)\mbox{-}{\rm mod} \ar[d]^{{\vert}_{{}_{\CX}}} \\ \mmod\CX \ar[rr]^{D} && \CX{\mbox{-}\rm{mod}} }\]
where $D$ in the rows are the corresponding  dualities for the dualizing $k$-varieties $\mmod \La$ and $\CX$,  and the functors in  columns are the corresponding restriction functors  on the subcategory $\CX$.

Since $D: \mmod (\mmod \La) \rt (\mmod \La)\mbox{-}{\rm mod}$ is a duality, it sends projective objects to the injective ones. Hence, in view of  the Gruson and Jensen's result, we may deduce that for every $M \in \mmod\La$, there exists a left $\La$-module $M'$ such that $D(\Hom_{\La}(-,M)) \simeq - \otimes_{\La}M'$. $M'$ is uniquely determined up to isomorphism, thanks to the fullness and  faithfulness of the functor $T$.\\

This isomorphism can be restricted to $\CX$, to induce the following isomorphism
\[D(\Hom_{\La}(-,M))\vert_{\CX} \simeq  - \otimes_{\La}M'\vert_{\CX}.\]
In the  case $M \in \CX$, this can be written more simply as \[D((-,M)) \simeq  - \otimes_{\La}M'\vert_{\CX}.\]
If $\La$, as a right module,  belongs to $\CX$, the by putting it in the above isomorphism we have $M'=D(M)$, here $D$ is the usual duality for $\mmod \La.$\\
Therefore, we have the following proposition:

\begin{proposition}\label{InjStable}
Let  $\CX$ be a subcategory of $\mmod \La$ which is a  dualizing $k$-variety subcategory and  containing $\La$, as a right $\La$-module. Then $\CX{\mbox{-}{\rm mod}}$ has enough injectives. Injective objects are those functors of the form $- \otimes M'\vert_{\CX}$, where $M'=D(X)$ for a unique object up to isomorphism $X$ in $\CX$.
\end{proposition}

For a subcategory $ \rm{proj }\mbox{-} \La \subseteq \CX \subseteq \mmod \La$, denote by  $\CX \mbox{-}^0 \rm{mod}$ consisting those functors of $\CX \mbox{-} \rm{mod}$ vanish on projective objects. Note that $\CX \mbox{-}_0 \rm{mod}$, defined in \ref{stableC},  including those  functors of $\CX \mbox{-} \rm{mod}$  vanish on injective modules which is different  of  $\CX \mbox{-}^0 \rm{mod}$ that whose objects vanish on projective modules.
Except in case of  when $\rm{prj} \mbox{-} \La= \rm{inj} \mbox{-} \La$, then $\CX \mbox{-}^0 \rm{mod}=\CX \mbox{-}_0 \rm{mod}.$

\begin{lemma}\label{AUSlemma}
Let $\CX$ be a  subcategory of $\mmod \La$ and containing $\rm{proj} \mbox{-} \La$. Then for each left module $M \in \La \mbox{-}\rm{mod}$, the functor $- \otimes_{\La}M\vert_{\CX}$ becomes an object in $\CX \mbox{-}\rm{mod}.$
\end{lemma}
\begin{proof}
The same argument as in \cite[ Lemma 6.1]{Au1} works here.
\end{proof}
\begin{proposition}\label{stableCOVAR}
Assume  $\CX$ is a covariantly finite  subcategory of $\mmod \Lambda $  and $ \rm{proj} \mbox{-} \Lambda \subseteq \CX $. Then  $\CX \mbox{-}^0 \rm{mod}\simeq \underline{\CX} \mbox{-} \rm{mod}.$ In particular, $\underline{\CX} \mbox{-} \rm{mod}$ is an abelian category with enough projectives.
\end{proposition}

\begin{proof}
The canonical functor $\pi:\CX \rt \underline{\CX}$   induces an equivalence $\pi^{*}$ between $\underline{\CX} \mbox{-} \rm{Mod}$ and $\CX \mbox{-}^0 \rm{Mod}$, the subcategory of $\CX \mbox{-} \rm{Mod}$ consisting  those of  functors which  vanish on projective modules. To find an equivalence between $\CX \mbox{-}^0 \rm{mod}$ and $\underline{\CX} \mbox{-} \rm{mod}$, we show that the restriction of $\pi^{*}$ on $\underline{\CX} \mbox{-} \rm{mod}$  will do it.

First to prove $\pi^{*} (F ) \in  \CX \mbox{-}^0 \rm{mod}$ for each $F $ in $\underline{\CX} \mbox{-} \rm{mod}$. Since $\pi^{*}$ is an exact functor,  it is enough to show  it in the  the case that $F=(\underline{X}, -)$ for some $X$ in $\CX.$ Note that $\pi^{*}((\underline{X}, -))$ is the same as $(\underline{X}, -)$ only considering it as a functor on $\CX$, that is,  for each $Y \in \CX, (\underline{X}, -)(Y)$ coincides with   $\Hom_{\La}(X, Y)$  modulo the morphisms factoring through a projective module. In view of the isomorphism $\underline{\Hom}_{\La}(M, -)\simeq \Tor_1(-, \rm{Tr}(M))$ for each $M \in \mmod \La,$ \cite[Proposition 2.2]{AR3}, we have
\[ 0 \rt (\underline{X}, -) \rt -\otimes_{\La}\Omega(\rm{Tr}(X))\vert_{\CX} \rt -\otimes_{\La}Q \vert_{\CX}\rt -\otimes_{\La}\rm{Tr}(X) \vert_{\CX}\rt 0\]
where $Q$ is in $\prj \La^{\op}.$ Since $\CX$ is covariantly finite  so $\CX \mbox{-} \rm{mod}$ is an abelian category, now by the above exact sequence and Lemma \ref{AUSlemma} we get our desired result.

To prove $\pi^*$ is a dense functor: Take $F \in \CX \mbox{-}^0 \rm{mod}$ and let $(Y, -) \rt (X, -) \rt F \rt 0$ be a projective presentation of $F$. Since $F$ vanishes on projective modules then it can be fitted in the following commutative diagram

\[\xymatrix{(Y, -) \ar[d] \ar[r] & (X, -)  \ar[r] \ar[d] & F \ar[r] \ar@{=}[d] & 0 \\
(\underline{Y}, -)  \ar[r] & (\underline{X}, -) \ar[r] & F \ar[r] & 0.}\]

Therefore, the lower row in the  above diagram says us we can consider $F$ as a functor on $\underline{\CX}$ which lies in $\underline{\CX} \mbox{-} \rm{mod}$. Further, it is mapped to $F$ by $\pi^*$ as a functor on $\CX$. 
\end{proof}

As an interesting consequence of the above result,  we in the following show that the property of  being dualizing $k$-variety of a functorially finite subcategory  can be transferred to  its stable category.

\begin{theorem}\label{StableDualizing}
Assume  $\CX$ is a functorially finite   subcategory of $\mmod \Lambda $ such that $ \rm{proj} \mbox{-} \Lambda \subseteq \CX $. Then the stable category $\underline{\CX}$ is also a dualizing $k$-variety.
\end{theorem}
\begin{proof}
Note that as mentioned in  \ref{DualityVaraity1}, $\CX$ is a dualizing $k$-variety.  Consider the following commutative diagram

\[ \xymatrix{  \CX \mbox{-} \rm{mod} \ar[rr]^{D} && \mmod \CX  \\
\CX \mbox{-}^0 \rm{mod} \ar@{^(->}[u] \ar[rr]^{D} && \rm{mod}_0 \mbox{-} \CX  \ar@{^(->}[u]\\
\underline{\CX} \mbox{-} \rm{mod} \ar[rr]^{D} \ar[u]^{\pi^{*}} && \rm{mod}\mbox{-} \underline{\CX} \ar[u]^{\pi^{**}} }\]
where $\pi^*$ and $\pi^{**}$ are the induced functors by the canonical functor $\pi:\CX \rt \underline{\CX}$, see Lemma \ref{stabel=mmodd} and  Proposition \ref{stableCOVAR}. Also. all  functors in horizontal rows are constructed just by sending functor $F$ in the corresponding domain to $DF$, that is for each object $X \in \CX,$  $DF(X)=\Hom_k(F(X), E)$, to be easy we use the symbol $D$ for all such functors. The functor $D$ in the middle is an equivalence since it is restriction of the functor $D$ in the  top, which is a duality  since $\CX$ is a dualizing $k$-variety. On the other hand, due to  Lemma \ref{stabel=mmodd} and Proposition \ref{stableCOVAR}, both functors $\pi^*$ and $\pi^{**}$ are equivalence. Now the commutativity diagram in  the above implies that the lowest contravariant functor $D$ is a duality. Hence,  $\underline{\CX}$ is a dualizing $k$-variety.
\end{proof}

In particular,  a special example of the above result is that  $\underline{\rm{mod }} \mbox{-}\La$ itself is a dualizing $k$-variety. It also can be seen in \cite{AR1}. َAnother immediate  consequence of the corollary is that $\mmod \underline{\CX}$ and $\underline{\CX} \mbox{-} \rm{mod}$ have enough injectives. With additional condition on $\CX$, that is, having split idempotents, we can conclude that  the  injective functors in $\mmod \underline{\CX}$ and $\underline{\CX} \mbox{-} \rm{mod}$ are those of $D(\underline{X}, -)$ and $D(-, \underline{X})$ for some $X \in \CX$, respectively.

In the following we  define  two other functors that are  basic parts of our next theorem.

\begin{construction}\label{tesorcon}
 For  a subcategory $\CX$ of $\mmod \Lambda$, $D(\CX)$ stands for the subcategory of  all left modules $D(X)$ for some $X \in \CX$. \\
 We define a functor $\Psi'_1: S_{D(\CX)}(\mmod \Lambda^{\rm{op}}) \rt \underline{\CX} \mbox{-} \rm{mod}$ respect to the subcategory $\CX$ as follows. Let $A \st{f} \rt B$ be an object  in $ S_{D(\CX)}(\mmod \Lambda^{\rm{op}}) .$ Consider the following  exact sequence
 $$0 \rt F \rt -\otimes_{\La}A \vert_{\CX} \st{-\otimes f}\lrt -\otimes_{\La} B \vert_{\CX}\rt -\otimes_{\La} \rm{Coker} (f) \vert_{\CX}\rt 0.$$
 Define $\Psi'_1(A \st{f}\rt B):= F$. Note that by Lemma \ref{stableCOVAR} we can think $F$ as an object in $\underline{\CX} \mbox{-} \rm{mod}.$ Similar to Construction \ref{FirstCoonstr},  $\Psi'_1$ can also be defined on the morphisms. Note that this turn  we are dealing with injective resolution in $\CX \mbox{-} \rm{mod}$. In a similar way, we can define functor $\Psi'_1:F_{D(\CX)}(\mmod \La^{\op}) \rt \underline{\CX} \mbox{-} \rm{mod} $, which sends an object $C \st{g} \rt D$ to the kernel of $-\otimes_{\La} \ker(g)\vert{\CX}  \rt -\otimes_{\La} C\vert_{\CX}$ in $\underline{\CX} \mbox{-} \rm{mod}$, or choosing a functor $G$ such that satisfying
$$0 \rt G \rt -\otimes_{\La} \ker(g) \vert_{\CX} \rt -\otimes_{\La} C \vert_{\CX} \st{-\otimes g} \lrt -\otimes_{\La} D \vert_{\CX}\rt 0.$$
\end{construction}

By use of the above construction we are able to prove new  equivalences.

\begin{theorem}\label{Thirdequivalence}
Let $\CX$ be a  covariantly finite subcategory of $\mmod \Lambda $ such that $ \rm{proj} \mbox{-} \Lambda \subseteq \CX $ and being
a dualizing  $k$-variety. Then

\begin{itemize}
\item[$(1)$]
 If $\CX$ is  closed under kernel of epimorphisms.
  Then the functor $\Psi_1$, defined in  construction \ref*{tesorcon},  respect to subcategory $\CX$, makes an equivalence of categories
$$ S_{D(\CX)}(\mmod \La^{\op})/ \mathcal{V}_1 \simeq \underline{\CX} \mbox{-} \rm{mod},$$
 where  $\mathcal{V}_1=\rm{add}\{(X \st{\rm{Id}} \rt X)\oplus (0 \rt X) \mid \text{for all} \  X \in \CX \}.$
  \item[$(2)$]
  If $\CX$ is  closed under cokernel and kernel  of monomorphisms and epimorphisms, respectively.
  Then the functor $\Psi'_1$, defined in  above construction, respect to the subcategory $\CX$, makes an equivalence of categories
$$ F_{D(\CX)}(\mmod \La^{\op})/ \mathcal{V'}_1 \simeq \underline{\CX} \mbox{-} \rm{mod},$$
 where  $\mathcal{V'}_1=\rm{add}\{(X \st{\rm{Id}} \rt X)\oplus (X \rt 0) \mid \text{for all} \  X \in \CX \}.$

 \end{itemize}

\end{theorem}

\begin{proof}
In general the proof of two statements is the same as  the proof of Theorem \ref{Thefirst} or Theorem \ref{Thesecound}. Let us give some points to complete our argument. For $F \in \underline{\CX} \mbox{-} \rm{mod}$ and using the equivalence in Proposition \ref{stableCOVAR} we can consider $F$ as a functor in $\CX \mbox{-} \rm{mod}$ which vanishes on projective modules. In view of Proposition \ref{InjStable} an injective resolution of $F$ is as follows:
$$0 \rt F \rt -\otimes_{\La}D(X) \st{-\otimes_{\La}D(f)} \lrt -\otimes_{\La}D(Y) \rt -\otimes_{\La}D(Z) \rt 0.$$

Note that by assumption on $\CX,$ $D(\CX)$ becomes closed under cokernel of monomorphisms and so $D(Z) \in D(\CX).$ Hence $D(X) \st{D(f)} \rt D(Y)$ becomes an object in $ S_{D(\CX)}(\mmod \La^{\op})$ such that $\Psi_1(D(X) \st{D(f)} \rt D(Y))\simeq F$. It is true for any $F$ in $\underline{\CX} \mbox{-} \rm{mod}$, and so $F$ is dense. About the kernel of $\Psi_1$: Of course, under functor $\Psi_1,$ the objects in $\mathcal{V}_1$ are sent to zero. Conversely, we assume $D(X) \st{D(f)} \rt D(Y)$ is an object in the kernel of $\Psi_1$. So we get short exact sequence $ 0 \rt  -\otimes_{\La}D(X) \st{-\otimes_{\La}D(f)} \lrt -\otimes_{\La}D(Y) \rt -\otimes_{\La}D(Z) \rt 0 $ in $\CX \mbox{-} \rm{mod}$. If we apply the duality $D=\Hom_{k}(-, E)$ on the short exact sequence, then we obtain the short exact sequence $0 \rt (-, Z) \rt (-, Y) \st{(-, f)}\lrt (-, X) \rt 0$ in $\mmod \CX.$ It follows that $f$ is a split epimorphism, or equivalently, $D(f)$ is a split monomorphism. So one can consider $D(f)$ as an object in $\mathcal{V}_1$. To obtain our equivalence we need to prove $\Psi_1$ is also full and objective, but it can be done in the same way of Theorem \ref{Thefirst}, the difference here is we are dealing with injective resolution instead of projective resolution in there. We finish the proof by pointing out that the additional assumption of being closed under kernel of epimorphisms in  dual part $(2)$ is necessary to prove that $\Psi'_1$ is dense.

\end{proof}
 The difference between  the  equivalences in  Theorem \ref{Thefirst} and \ref{Thesecound} with  the  equivalences in  the above theorem is that in the latter left and right modules simultaneously are involved.
\begin{construction}\label{ConstDul}
Suppose that $\CX$ is a functorially finite subcategory of $\mmod \La.$ {\bf First:} If $\CX$ conatians $\rm{prj} \mbox{-} \La$ and  $\rm{inj} \mbox{-} \La$ and is  closed under cokernel and kernel  of monomorphisms and epimorphisms, respectively,  then by Theorem \ref{Thefirst} and Theorem \ref{Thirdequivalence}, we have the following equivalence 
\[\Theta:\mmod \underline{D(\CX)} \st{\Psi}\rt S_{D(\CX)}(\mmod \La^{\op})/ \mathcal{V} \st{\rm{Cok}} \rt F_{D(\CX)}(\mmod \La^{\op})/ \mathcal{V'}_1 \st{\Psi_1'} \rt \underline{\CX} \mbox{-} \rm{mod}. \]
Set $$\rm{Tr}_{\CX}:\underline{D(\CX)} \st{Y}\rt \rm{Prj} \mbox{- }\mmod \underline{D(\CX)}  \st {\Theta\mid}\rt  \rm{Prj} \mbox{-}\underline{\CX}\mbox{-} \rm{mod} \st{Y'} \rt \underline{\CX},$$ where $Y$ and $Y'$ are functors coming from the  Yoneda Lemma. Two functors $Y$ and $\Theta \mid$ are equivalence and the last functor $Y'$ is a duality. So  $\rm{Tr}_{\CX}$ is a duality.

{\bf Second:} Let $\rm{prj} \mbox{-} \La \subseteq \CX$ be closed under kernel of epimorphisems. Then again by Theorem \ref{Thirdequivalence} and Theorem \ref{Thesecound}, we obtain  the following duality
\[ \Theta': \underline{\CX} \mbox{-} \rm{mod} \st{\Psi_1^{-1}}\rt S_{D(\CX)}(\mmod \La^{\op})/ \mathcal{V}_1 \st{\rm{Cok}} \rt F_{D(\CX)}(\mmod \La^{\op})/ \mathcal{V'} \st{\Psi'}\rt \overline{D(\CX)} \mbox{-} \rm{mod}.\]
\end{construction}

 Let  $M \in \mmod \La^{\rm{op}}$ and choose  $0 \rt M \st{i} \rt I \rt \Omega^{-1}(M) \rt 0$, where $I$ is the injective envelope of $M$, the functor $\mathcal{N}_{M}:\underline{\CX} \mbox{-} \rm{mod} \rt \mathcal{A}b$ for subcategory $\CX$ contating $\rm{prj} \mbox{-}\La$ and covariantly finite is defined 
 $$\mathcal{N}_M: \rm{Ker}((-\otimes_{\La}M ) \st{-\otimes_{\La}i} \lrt (-\otimes_{\La} I)), $$
by taking kernel in $\CX \mbox{-} \rm{mod}.$ We also can define similar notations for a $M \in \mmod \La$ and $\CX \subseteq \mmod \La^{\rm{op}}.$

In the following we will list some properties of the functors are defined in Construction \ref{ConstDul} which  can be verified most of them in a straightforward way.  So we leave the proof to the reader.
\begin{proposition}\label{Vartens}
For each $X \in \CX$, we have the following properties:
\begin{itemize}
\item[(i)] Let  $\CX$ be the same as the first part of Construction \ref{ConstDul}. Then
\begin{itemize} 
\item[(1)]$\Theta((-, \underline{D(X)}))=\Tor_1(-, D(X)) \mid_{\CX}\simeq (\underline{\rm{Tr}_{\CX}(D(X))}, -);$
\item[(2)] $\Theta(\Ext^1(-, D(X))\mid_{D(\CX)})=\mathcal{N}_{D(X)};$
\item[(3)] There exists a morphism $s:\rm{Tr}D(X) \rt \rm{Tr}_{\CX}(D(X))$ so that for any morphism $g:\rm{Tr}D(X) \rt Y$ in $\underline{H(\La)}$ with $Y $ in $\CX$, there is a unique $h:\rm{Tr}_{\CX}(D(X)) \rt Y$ such that $h \circ s=g$ in $\underline{H(\La)}$. In particular, $s$ is isomorphic to the   minimal $\underline{\CX}$-left approximation of $\rm{Tr}(D(X))$ in $\underline{H(\La)}.$

\end{itemize}
\item[(ii)] Let $\CX$ be the same as the second part of Construction \ref{ConstDul}. Then
\begin{itemize}
\item[(1)] $(\Theta')^{-1}((\overline{D(X)}, -))=\mathcal{N}_{D(X)}.$
\item[(2)] $(\Theta')^{-1}(\Ext^1(D(X),-)\mid_{D(\CX)})=\Tor_1(-, D(X))\mid_{\CX}.$
\end{itemize}
\item[(iii)] Assume $\CX$ holds  either the conditions of the first  or the second  part of construction \ref{ConstDul}, and moreover closed under direct summand.  If $F$ is an injective functor in $\underline{\CX} \mbox{-} \rm{mod},$ then $F \simeq \mathcal{N}_{D(X)}$ for some $X \in \CX.$
\end{itemize}
\end{proposition}
\begin{proof}

Only for $(i)(3)$ use the isomorphism $\underline{\rm{Hom}}_{\La}(\rm{Tr}D(X), -)\simeq
   \Tor_1(-, D(X))$ and then with help of $(i)(1)$ we have   
 
 \[\upsilon:\underline{\rm{Hom}}_{\La}(\rm{Tr}D(X), -)\mid_{\CX}\simeq  (\underline{\rm{Tr}_{\CX}(D(X))}, -). \] 
 Take $ s $ the image of $1_{\rm{Tr}_{\CX}(D(X))}$ under the evaluation of $\upsilon $ on $\rm{Tr}_{\CX}(D(X)).$
 \end{proof}
\begin{remark}
By Proposition \ref{Vartens}, we can obtain other variations of Auslander's conjecture of subsection \ref{AusConj}. For example, for a subcategory $\CX $ as Proposition \ref{Vartens} (iii), any direct summand of $\Tor_1(-, D(X)) \mid_{\CX}$, for some $X \in \CX$ is again in the $\rm{Tor}$-form, or, similarly, for the functor in the form of $\mathcal{N}_{D(X)}$ for some $X \in \CX.$
\end{remark}

\begin{example}\label{trans}
Suppose $\CX$ is $\mathcal{P}^{<\infty}(\La)$ for a Gorenstein  algebra $\La,$ see Example \ref{Examplefirst}. Clearly, $\mathcal{P}^{<\infty}(\La)$ satisfies the conditions of both parts of Construction \ref{ConstDul}.\\
 We can see that $D(\mathcal{P}^{<\infty}(\La))=\mathcal{P}^{<\infty}(\La^{\rm{op}}).$ Hence by the construction we have dualities $\rm{Tr}_{\mathcal{P}^{<\infty}(\La)}:\underline{\mathcal{P}^{<\infty}(\La^{\rm{op}})} \rightleftarrows \underline{\mathcal{P}^{<\infty}(\La)}:\rm{Tr}^{-1}_{\mathcal{P}^{<\infty}(\La)} $, for simplicity, this pair  of quasi-inverse equivalences are shown by $\rm{Tr}_{\mathcal{P}^{<\infty}}$. Another example we can add is whenever $\CX=\mmod \La$, $\rm{Tr}_{\mmod \La}$ is nothing else than the usual Transpose functor $\rm{Tr}$, defined by Auslander. Let us point out $\rm{Tr}_{\mathcal{P}^{<\infty}}$, in general, is not necasserly to be the restriction of the $\rm{Tr}$ over $\mathcal{P}^{<\infty}(\La)$. It might be interesting to compute $\rm{Tr}_{\mathcal{P}^{<\infty}}$, as we will do it for a specific case in the following.

\end{example}
Let us introduce some notations which  will be useful to state our next results.
\begin{definition}\label{Locallyproje}
An object $P=(P_1 \st{f}\rt P_2)$ in $H(\Lambda)$, or $\rm{rep}(\mathbb{A}_2, \La)$, is called locally projective representation if both $P_1$ and $P_2$ are in $\rm{prj}\mbox{-}\La.$
The  subcategory of $H(\Lambda)$ consisting of all locally projective representations will be denoted by $\mathcal{P}(\La).$ Indeed, by definition $\mathcal{P}(\La)=\mathbb{A}_2(\rm{prj}\mbox{-}\La)$, see example \ref{Examplefirst}. Locally injective representations  can be defined similarly, denoted by $\mathcal{I}(\La)$.
\end{definition}

For   $N \in \mmod \La$, let  $\rm{IMin}(N)$ stand for    the representation $I_0 \st{g} \rt I_1$ so that there exists exact sequence  $ 0 \rt N \st{\pi} \rt I_0 \st{g}   \rt  I_1   $ in which $I_0$ and $I_1$ are  injetive envelope  of $N$ and $\rm{Im}g$, respectively. Obviously,  we have a canonical morphism in $H(\La)$, say $\Theta_N$, from $(N \rt 0 ) $ to $(I_0 \st{g} \rt I_1),$ if we consider them as objects in $H(\La).$ The notation $\rm{PMin}(N)$ can be defined dually by using  of projective cover.

If $\La$ is a self-injective algebra, then $T_2(\La)$ is a $1$-Gorenstein algebra. If we  use the equivalence $\mmod T_2(\La)\simeq H(\La),$ then $\mathcal{P}^{<\infty}(T_2(\La)) \simeq \mathcal{P}(\La)$. Then by Example \ref{trans}, we have duality $\rm{Tr}_{\mathcal{P}(\La)}: \underline{\mathcal{P}(\La^{\rm{op}})} \rt \underline{\mathcal{P}(\La)}.$ In the following we describe this duality on each object. 
\begin{proposition}\label{TransP} Keep the above notations for a self injective algebra $\La.$ Let $P=(P_1 \st{f}\rt P_2)$ be in $ \mathcal{P}(\La^{\rm{op}})$. and set $P_1=Q_1\oplus Q_2$ such that  $f \mid_{Q_1}= 0$ and  the largest summand with this property. Then we have $\rm{Tr}_{\mathcal{P}(\La)}(P)\simeq\rm{IMin}(\rm{Tr}(\rm{Coker}(f)) \oplus Q^*_1 )$ in $\underline{\mathcal{P}(\La)}.$ 

\end{proposition}
\begin{proof}
For any locally projective representation $P=(P_1 \st{f}\rt P_2)$, we have the following short exact sequence
{\footnotesize \[ \xymatrix@R-2pc {  &  ~ 0 \ar[dd]~  &  & P_1\ar[dd]^{[{0~~1}]}~~  & &  P_1 \ar[dd]^{f} \\ (\dagger) _{\ \ \ \ \ \ \ \ \  \  } 0 \ar[r] &  _{ \ \ \ \ } \ar[rr]_{[-f~~1]} & &_{\ \ \ \ \ } \ \ \ar[rr]^{1} _{[1~~f]^t}{\ \ \ \ } & &~~  _{\ \ \ \ \ }\ar[r] & 0. \\ & P_1 & & P_2\oplus P_1 & & P_2 }\]}
which   in fact is  a projective resolution for $P=(P_1 \st{f}\rt P_2)$ but not minimal in general. we will use this short exact sequence to compute the Auslander transpose $\rm{Tr}_{\mathbb{A}_2}(P)$ in $H(\La).$ Applying $(-, \mathbb{A}_2)$ on $(\dagger)$, then we get 
{\footnotesize \[ \xymatrix@R-2pc {  &  ~ P^*_2 \ar[dd]^{[{0~~1}]}~  &  & ~~ P^*_1\ar[dd]^{1}~~  & \rm{Tr}(\rm{Coker} (f)) \oplus Q^{*}_1\ar[dd] \\  &  _{ \ \ \ \ }\ \ \ \ \ \ \   \ar[rr]^{f^*}_{[(-f)^*~~1]^t}  & & \ \ \ \ \ \  \ar[r]&  _{\ \ \ \ \ }\ar[r] & 0. \\ & P^*_2 \oplus P^*_1 & &~~ P^*_1 & 0  }\]}

Note that here the  contravariant functor  $(-, \mathbb{A}_2)$ obtained by considering equivalence between $\mmod T_2(\La)$ and $\rm{rep}(\La, \mathbb{A}_2)$ as follows: 
\[\xymatrix{\rm{rep}(\La, \mathbb{A}_2) \ar[rrr]^{(-, \mathbb{A}_2)} \ar[d]^{\simeq} & & & \rm{rep}(\La^{\rm{op}}, \mathbb{A}_2) \ar[d]^{\simeq} \\ \mmod T_2(\La) \ar[rrr]^{\Hom_{T_2(\La)}(-, T_2(\La))} & & & \mmod T_2(\La^{\rm{op}}). }\]

 To take cokernel in above: $(1)$ We use  that $[(-f)^*~~1]$ is a split epimorphism. $(2)$  If we  consider $P_1 \st{f} \rt P_2$ as an object in $H(\La^{\rm{op}})$, then  we can decompose $P$ as $P\simeq (Q_1 \rt 0)\oplus (Q_2 \st{f'} \rt P'_2) \oplus (0\rt Q)$, where $Q_2 \st{f'} \rt P'_2$ is a minimal presentation of $\rm{Coker}(f)$ and   $Q \in \rm{prj} \mbox{-}\La^{\rm{op}}$. Now due to definition of transpose we obtain $\rm{Coker}(f^*)\simeq \rm{Tr}(\rm{Coker}(f))\oplus Q^*_1.$ By Proposition \ref{Vartens}, it is enough  to find a minimal left $\mathcal{P}(\La)$-approximation for $(\rm{Tr}(\rm{Coker}(f))\oplus Q^*_1 \rt 0)$ in $\underline{H(\La)}$.  Consider the canonical morphism $\Theta_{\rm{Tr}(\rm{Coker}(f))\oplus Q^*_1}$ in $H(\La)$, that is, in fact,  a minimal left $\mathcal{P}(\La)$-approximation  of $(\rm{Tr}(\rm{Coker}(f))\oplus Q^*_1 \rt 0)$ in $H(\La)$. Using this fact that $\Hom_{H(\La)}((M\rt 0), (N \st{1}\rt N)\oplus (0 \rt Z))$, for any $M, N$ and $Z$ in $\mmod \La$, we can deduce $\Theta_{\rm{Tr}(\rm{Coker}(f))\oplus Q^*_1}$ is a minimal left  $\underline{\mathcal{P}(\La)}$-approximation of $(\rm{Tr}(\rm{Coker}(f))\oplus Q^*_1 \rt 0)$ in $\underline{H(\La)}.$ So it completes the proof.
\end{proof}

\section{The Auslander-Reiten duality  for (co)resolving and functorially finite subcategories}
In this section  we will show the existence of Ausalnder-Reiten duality  for some sufficiently nice subcategory. First we start with following construction.

\begin{construction}\label{Constructionsec}
Assume $\CX$ satisfying the assumptions of Theorem \ref{Thirdequivalence} $ (2)$.
The usual duality $D$ for $\rm{mod }\mbox{-}\Lambda$ gives $ \overline{D(\CX)} \simeq (\underline{\CX})^{\op}$. It induces
\[\overline{D(\CX)} \mbox{-} \rm{mod}\simeq (\underline{\CX})^{\op} \mbox{-} \rm{mod} \simeq \mmod \underline{\CX}.\]
The following composition

$$\underline{\CX} \mbox{-} \rm{mod} \st{(\Psi'_1)^{-1}}  \lrt F_{D(\CX)}(\mmod \La^{\op})/ \mathcal{V}'_1  \st{\Psi'}\rt  \overline{D(\CX)} \mbox{-} \rm{mod} \st{\simeq} \lrt \rm{mod} \mbox{-}\underline{\CX},$$
here $\Psi'$, defined in Theorem \ref{Thesecound}, respect to the  subcategory $D(\CX)$ of $\mmod \La^{\op}$, gives us  a duality. Denote this duality by $\Phi_1$.
\end{construction}
 In the sequel, we will show that our duality in above construction is isomorphic to  the usual duality $D$ for when $\underline{\CX}$ considered as a dualizing $k$-variety, see Theorem \ref{StableDualizing}.
\begin{proposition}\label{DualityDec}
Let $\CX$ be functorially finite  subcategory of $\mmod \Lambda $ such that  $ \rm{proj} \mbox{-} \Lambda \subseteq \CX $, and  with   split idempotents (or closed under direct summand). Further, assume $\CX$  is closed under kernel and cokernel of epimorphisms and monomorphisms, respectively. Let $\Phi_1$ be the duality  respect to the subcategory $\CX$,  defined in Construction \ref{Constructionsec}, and $D:\underline{\CX} \mbox{-} \rm{mod} \rt \mmod \underline{\CX}$ be the usual duality for the dualizing variety $\underline{\CX}$. Then $\Phi_1\simeq D.$
\end{proposition}
\begin{proof}
It is enough to show that $\Phi_1 \circ D\simeq1,$ here $D:\mmod \underline{\CX} \rt \underline{\CX} \mbox{-} \rm{mod}$ denotes the quasi-inverse of $D$ in the statement, the same notation used to show them. For each $X$ in $\CX$, we have exact sequence
$$0 \rt (-, \Omega(X)) \rt (-, P) \rt (-, X) \rt (-, \underline{X}) \rt 0,$$
in $\mmod \CX.$
By applying $D=\Hom_{k}(-, E)$ on  the above exact sequence we find the following exact sequence
$$0 \rt D(-, \underline{X}) \rt D(-, X) \rt D(-, P) \rt D(-, \Omega(X)) \rt 0$$
in $\CX \mbox{-} \rm{mod}.$
Using known isomorphism $D(-, M)\simeq -\otimes_{\Lambda}D(M)$ for each $M,$ and whose restriction on $\CX$,  then we get
$$ (*)  \quad   0 \rt D(-, \underline{X}) \rt -\otimes_{\Lambda}D(X)\vert_{\CX} \rt -\otimes_{\Lambda}D(P)\vert_{\CX} \rt -\otimes_{\Lambda}D(\Omega(X))\vert_{\CX} \rt 0$$
in $\CX \mbox{-} \rm{mod}.$
 On the other hand, we have exact sequence
$$ (**)  \quad 0 \rt (D(\Omega(X)), -) \rt (D(P), -) \rt (D(X), -) \rt (\overline{D(X)}, -) \rt 0$$
in $D(\CX) \mbox{-} \rm{mod},$ induced by the exact sequence $0 \rt D(X) \rt  D(P) \rt D(\Omega(X)) \rt 0$ in $\mmod \La^{\op}.$

Let us denote by $\Upsilon$ the last equivalence appeared to define duality $\Phi_1$. Now we use the above short exact sequences to prove our claim. Before than since $\Phi_1 \circ D$ is an exact functor it is enough to show that it is isomorphic to identity functor  on projective functors, that is, $\Phi_1 \circ D ((-, \underline{X}))\simeq (-, \underline{X})$ for each $X \in \CX.$ Note that since idempotents split in $\CX$ then all projective functors in $\mmod \CX$ are  representable functors. Consider
 \begin{align*}
 \Phi_1 \circ D ((-, \underline{X}))&= \Upsilon \circ \Psi' \circ (\Psi'_1)^{-1} \circ D ((-, \underline{X})) \\
 & \stackrel{(*)}{\simeq} \Upsilon \circ \Psi' (D(P) \rt D(\Omega(X))) \\
 & \stackrel{(**)}{\simeq} \Upsilon (\overline{D(X)}, -)\\
 & \hspace{1 mm} \simeq (-, \underline{X}).
 \end{align*}
\end{proof}

Let us give some examples holding conditions of Proposition \ref{DualityDec}.

\begin{example}
Trivially $\mmod \La$ as a subcategory of itself holds the assumptions. The subcategories in Example \ref{Examplefirst} satisfy the  conditions we need in Proposition \ref{DualityDec} except here we no need to assume the condition of containing injective modules.  In contrary of \ref{Examplefirst} for the subcategory $\mathcal{P}^{< \infty}(\La)$, no need to assume that here $\La$ is a Gorenstein algebra. Only by a result of Krause and Solberg \cite[Corollary 0.3]{KS} we need to see for which algebras $\mathcal{P}^{< \infty}(\La)$ is contravariantly finite in $\mmod \La.$ There is a very wide of algebras with $\mathcal{P}^{< \infty}(\La)$ is contravariantly finite, see e.g. \cite{AR2}.
\end{example}

This description of $D$ motivate us to give an another proof for the existence of Auslander-Reiten duality as follows:
For $M$ in $\mmod \La$, we have
$$0 \rt (-, M) \rt (-, I) \rt (-, \Omega^{-1}(M)) \rt \Ext_{\La}^1(-, M) \rt 0$$
in $\mmod \CX$, and also,
$$0 \rt \Tor_1(-, D(M)) \rt -\otimes_{\La}D\Omega^{-1}(M) \rt -\otimes_{\La} D(I) \rt -\otimes_{\La} D(M) \rt 0$$
and now by using isomorphism $\Tor_1(-, N)\simeq \Hom(\underline{\rm{Tr}(N)}, -)$, for each $N$ in $\La \mbox{-} \rm{mod}$

$$0 \rt  \Hom_{\La}(\underline{\rm{Tr}(D(M))}, -) \rt -\otimes_{\La}D\Omega^{-1}(M) \rt -\otimes_{\La} D(I) \rt -\otimes_{\La} D(M) \rt 0.$$
Then by using the above exact sequences and our description of $D$, Proposition \ref{DualityDec}, we obtain
$$ D(\Ext_{\La}^{1}(-, M)) \simeq \Phi_1(\Ext_{\La}^1(-, M))\simeq \underline{\Hom}_{\La}(\rm{Tr}D(M), -).$$
This formula is the same as the introduced Auslander-Reiten duality in the Introduction. As we have  seen  the Auslander-Reiten duality can be proved completely in a functional approach.  
 \begin{theorem}\label{AR-duality}

Let $\CX \subseteq \mmod \La$ be resolving-coresolving,  functorially finite and closed under direct summand. Then $\CX$ has  Auslander-Reiten duality, i.e, there exist an equivalence functor $\tau_{\CX}:\underline{\CX} \rt \overline{\CX}$  such that for any $X$ and $Y$, we have the following natural isomorphism in both variables;
\[D\underline{\Hom}_{\La}(X, Y)\simeq \Ext_{\La}^1(Y, \tau_{\CX}(X)).\]

\end{theorem}
\begin{proof}
The assumptions on $\CX$ are exactly  all conditions we need for Proposition \ref{DualityDec} and \ref{enoughinj}. Hence these assumptions guarantees the  existence of dualities $D=\Phi_1:\underline{\CX} \mbox{-} \rm{mod} \rt \mmod \underline{\CX}$, see Construction \ref{Constructionsec} or Theorem \ref{StableDualizing},  and $\Phi: \mmod \underline{\CX} \rt \overline{\CX} \mbox{-} \rm{mod}$, see subsection \ref{Firstconstuction}, respect to the subcategory $\CX$. Consider the composition of $D$ and $\Phi$
$$\underline{\CX} \mbox{-} \rm{mod} \st{D}\rt \mmod \underline{\CX} \st{\Phi} \rt \overline{\CX} \mbox{-} \rm{mod} $$
which gives us an equivalence  between abelian categories $\underline{\CX} \mbox{-} \rm{mod}$ and $\overline{\CX} \mbox{-} \rm{mod}.$ This composition $ \Phi \circ D$ can be restricted on projective functors and then by equivalences $(\underline{\CX})^{\op}\simeq \rm{Prj} \mbox{-} (\underline{\CX} \mbox{-} \rm{mod})$ and $(\overline{\CX})^{\op}\simeq \rm{Prj} \mbox{-} (\overline{\CX} \mbox{-} \rm{mod})$ induced by the  Yoneda Lemma, we find an equivalence $\underline{\CX} \rt \overline{\CX}$, denote it by $\tau_{\CX}.$ The following diagram explains more what we followed
 \[ \xymatrix{ \underline{\CX} \mbox{-}\rm{mod} \ar[r]^D & \mmod \underline{\CX} \ar[r]^{\Phi} & \overline{\CX} \mbox{-}\rm{mod} \\
    \underline{\CX} \ar[u]^{Y} \ar[rr]^{\tau_{\CX}} && \overline{\CX} \ar[u]_{Y'}, }\]
where $Y$ and $Y'$ are the corresponding contravariant functors by the  Yoneda Lemma. Let $X$ be a module in $\CX,$
 \begin{align*}
      \Phi \circ D ((\underline{X}, -))&=(\overline{\tau_{\CX}(X)}, -) \\
      D((\underline{X}, -))  &= \Phi'((\overline{\tau_{\CX}(X)}, -))  \\
       &\simeq\Ext_{\La}^1(-, \tau_{\CX}(X))\vert_{\CX}.
    \end{align*}
the last equivalence can be inferred by a dual argument in the proof of Theorem \ref{AusConj}. Therefore, in this way we prove a relative version of the  Auslander-Reiten duality respect to such subcategories.
\end{proof}

In particular in above theorem, if $\CX=\mmod \La$, then by using  \cite[Hilton-Rees Theorem]{HR} and the original Auslander-Reiten formula, one can deduce that $\tau_{\mmod \La} \simeq D\rm{Tr}=\tau_{\La}.$  \\
We obtain the following as a consequence of Theorem \ref{AR-duality} together with   the original Auslander-Reiten duality.
\begin{corollary}
Let $\CX \subseteq \mmod \La$ be the same as Theorem \ref{AR-duality}. Then for each $X$ and $Y$ in $\CX,$ $\Ext_{\La}^1(Y, \tau(X))\simeq \Ext_{\La}^1(Y, \tau_{\CX}(X)).$
\end{corollary}
By help of a result of Raymundo Bautista \cite{Ba}  in the following example we compute relative Auslander-Reiten translation for a certain subcategory of morphism category over a self-injective algebra.
\begin{example}
For an arbitrary artin algebra $\La$ let $\mathcal{P}(\La)$, resp. $\mathcal{I}(\La)$, be the category of locally projective, resp. injective,  representation on $\mathbb{A}_2$, see Definition \ref{Locallyproje}. The functor $-\otimes_{\La}D(\La)$, known as Nakayam functor and  usually denoted by $\mathcal{N}$,  gives an equivalence between $\rm{prj} \mbox{-} \La$ and $\rm{inj } \mbox{-} \La$. Thus this functor induces an equivalence $F:\mathcal{P}(\La) \rt \mathcal{I}(\La)$. Since $F$ sends the subcategory of projective objects in $H(\La)$, clearly is contained in  $\mathcal{P}(\La)$, into $\mathcal{C}=\rm{add} \{(I\st{1}\rt I)\oplus (0 \rt I) \mid I \in  \rm{inj}\mbox{-} \La \} \subseteq \mathcal{I}(\La)$, then we have equivalence $\hat{F}:\underline{\mathcal{P}(\La)} \rt \mathcal{I}(\La)/I_{\mathcal{C}}$. Consider the functor $\rm{\mathcal{C}ok}:\mathcal{P}(\La) \rt \mmod \La $ given by $\rm{\mathcal{C}ok}(P_1 \st{f} \rt P_2)=\rm{Coker}(f)$, and also functor $\rm{\mathcal{K}er}: \mathcal{I}(\La) \rt \mmod \La$ given by $\rm{\mathcal{K}er}(I_1 \st{g} \rt I_2)=\rm{Ker}(g).$ The functor $\rm{\mathcal{C}ok}$, resp. $\rm{\mathcal{K}er}$, induces an equivalence between $\mathcal{P}(\La)/I_{\mathcal{C}'}$, resp. $\mathcal{I}(\La)/I_{\mathcal{C}}$ and $\mmod \La,$ where $\mathcal{C}'=\rm{add} \{(P\st{1}\rt P)\oplus (P \rt 0) \mid P \in  \rm{prj}\mbox{-} \La \}$. Denote these two resulting  equivalences by $\overline{\rm{\mathcal{C}ok}}$ and $\overline{\rm{\mathcal{K}er}}$, respectively. Thus we have the following 
equivalence
\[ \hat{A}:=\underline{\mathcal{P}(\La)} \st{\hat{F}}\rt  \mathcal{I}(\La)/I_{\mathcal{C}} \st{\overline{\rm{\mathcal{K}er}}} \lrt \mmod \La \st{(\overline{\rm{\mathcal{C}ok}})^{-1}} \lrt \mathcal{P}(\La)/I_{\mathcal{C}'}. \]
Let us explain the above notations in the context of exact categories as done in \cite{Ba}. Since $\mathcal{P}(\La)$ is a subcategory of $H(\La)$ closed under extension, so it gets an exact structure inherited by abelian structure of $H(\La).$  In fact, $\mathcal{P}(\La)/I_{\mathcal{C}'}$, resp. $\underline{\mathcal{P}(\La)}$, defined above, are stable category of $\mathcal{P}(\La)$ modulo of injective, resp. projective objects in the corresponding exact category.
In \cite{Ba} was proved by  Raymundo Bautista  for any $P$ and $Q$ in $\mathcal{P}(\La)$, there is a natural isomorphism
\[ (\dagger) \ \ \ \  \  \ \Ext_{H(\La)}^1(P, Q)\simeq D\underline{\rm{Hom}}_{H(\La)}(Q, \hat{A}(P)).    \]
Suppose $\La$ is self-injective. Then $\mathcal{P}(\La)=\mathcal{I}(\La)$ and also $\mathcal{P}(\La)/I_{\mathcal{C}'}= \overline{\mathcal{P}(\La)}.$ Recall that $\underline{\mathcal{P}(\La)}$, resp. $\overline{\mathcal{P}(\La)}$, are the stable category of $\mathcal{P}(\La)$ modulo projective, resp. injective, objects in $H(\La)$. In this case, $\mathcal{P}(\La)$ as a subcategory of  $H(\La)$ satisfies the assumptions of Theorem \ref{AR-duality}. So there is the equivalence functor $\tau_{\mathcal{P}(\La)}: \underline{\mathcal{P}(\La)} \rt \overline{\mathcal{P}(\La)}$ such that for given $P$ and $Q$ in $\mathcal{P}(\La)$
\[(\dagger \dagger) \ \ \ \  \  \  \Ext_{H(\La)}^1(P, Q)\simeq D\underline{\rm{Hom}}_{H(\La)}(Q, \tau_{\mathcal{P}(\La)}(P)).   \]
Therefore, by $(\dagger)$ and $(\dagger \dagger)$ we can  conclude that $\underline{\rm{Hom}}_{H(\La)}(Q, \hat{A}(P))\simeq\underline{\rm{Hom}}_{H(\La)}(Q, \tau_{\mathcal{P}(\La)}(P)),$ but it implies $\hat{A}\simeq \tau_{\mathcal{P}(\La)},$ by Yoneda Lemma. Moreover, it can be verified by the definition of this equivalence in conjunction with duality $\rm{Tr}_{\mathcal{P}(\La)}$ in Proposition \ref{TransP}   that $\tau_{\mathcal{P}(\La)}\simeq D\rm{Tr}_{\mathcal{P}(\La)}.$ Here, similar situation for the relation between  usual Auslander-Reiten and usual transpose, $\tau\simeq D\rm{Tr}$, happens for this certain case.
\end{example}
We keep all notations in the above example in the rest of this section. In fact,  in the above example, we compute $\tau_{\mathcal{P}^{<\infty}(T_2(\La))}$ for algebra $T_2(\La)$ whenever $\La $ is a self-injective algebra. Since in this case $T_2(\La)$ is  an $1$-Gorenstein 
algebra and then $\mathcal{P}^{< \infty}(T_2(\La))\simeq \mathcal{P}(\La)$ in respect of identification $\mmod T_2(\La)\simeq H(\La).$ we can obtain more information in the following for $\tau_{\mathcal{P}^{<\infty}(T_2(\La))}$, or equally $\tau_{\mathcal{P}(\La)},$ for more specific self-injective algebras.
\begin{corollary}
Let $\La$ be a self-injective.
\begin{itemize}
\item[$(i)$] If $\tau_{\La}$ coincide with identity functor, then $\tau_{\mathcal{P}^{<\infty}(T_2(\La))},$ or equally $\tau_{\CP(\La)},$ is also identity functor.
\item[$(ii)$] Assume that $\tau_{\La}$ coincides with $\Omega^2,$  e.g.  any symmetric algebra. Let $P_1\st{f}\rt P_0$ be in $\CP(\La)$. Consider the following exact sequence
\[P_3 \st{g}\rt P_2 \st{h}\rt P_1 \st{f}\rt P_0 \rt \rm{Coker}(f)\rt 0\]
which all terms, possibly exept $\rm{Coker}(f)$ are projective objects. Then $\tau_{\CP(\La)}(f)\simeq g$ in $\overline{\CP(\La)}.$
\end{itemize}
\end{corollary}
\begin{proof}
$(i)$  We prove the statement for the equivalent case, that is, $\tau_{\mathcal{P}(\La)}\simeq 1_{\mathcal{P}}.$ Let $P_1 \st{f}\rt P_0 \in \CP (\La),$ we can assume that $\rm{Coker}(f)$ has no 
projective direct summand in $\mmod \La.$ Then by our assumption we have the following exact sequence
$$0 \rt \tau(\rm{Coker}(f))\simeq \rm{Coker}(f) \rt \CN(P_1) \st{\CN(f)} \rt \CN(P_0) \rt \CN(\rm{Coker}(f)) \rt 0. $$

Notice that we have used here  the fact which $\CN(M)\simeq \Omega^{-2}(\tau(M)),$ for any module without projective summand. Now by our computation in the above example, we
 can get $\tau_{\CP(\La)}(P_1 \st{f}\rt P_0)=P_1 \st{f} \rt P_0.$ So we are done. The proof of part $(ii)$ is similar to $(i)$, so we leave it to the reader.
\end{proof}

 In sequel, we shall try to explain a procedure to  compute $\tau_{\mathcal{P}^{<\infty}(\La)}$ when $\La$ is an arbitrary $1$-Gorenstein algebra. We refer to \cite{GLS} for many examples which show that $\tau_{\CP(\La)^{<\infty}}$ can be  a proper translation, i.e. not to be isomorphic to usual Auslander-Reiten translation in general. \\
  Let us first recall some facts from \cite{RS3}. Let $(A \st{f}\rt B)$ be an arbitrary object in $H(\La)$. Assume that $e':\rm{Ker}(f) \rt I(\rm{Ker}(f))$ be an injective envelope and choose an extension $e: A \rt I(\rm{Ker}(f)) $ of $e'.$ Consider the following object in $H(\La)$
  $$\rm{Mimo}(f)=[f, e]:  \ A \rt B\oplus I(\rm{Ker}(f)).$$
  
Moreover, there is a canonical   morphism $(1_A, [1_, 0]): \rm{Mimo}(f) \rt f $ in $H(\La).$ Let $\mathcal{S}(\La)$ stand for the subcategory of all monomorphisms in $H(\La).$  It is shown in Proposition 2.4 of \cite{RS3} the morphism $\rm{Mimo}(f) \rt f$ is a minimal right approximation of $f$ in $\mathcal{S}(\La).$ This fact plays an important role to prove our next result. By applying cokernel on $(1_A, [1_, 0])$, we obtain morphism $h:\rm{Coker}(\rm{Mimo}(f)) \rt \rm{Coker}(f)$. By use of a general fact, one can uniquely decompose $h$ as $h=[h_1, h_2]: \rm{Coker}(\rm{Mimo}(f))=N_1\oplus N_2 \rt \rm{Coker}(f)$ such that $h_1$ is right minimal and $h_2=0.$ Let us denote $N_1$ by $\overline{\rm{Coker}}(\rm{Mimo}(f)).$ 

We need also the following lemma.
\begin{lemma}\cite[Proposition 3.3]{AR2}\label{Applic}
Let $\CX$ be a resolving, contravariently finite subcategory of $\mmod \La.$ Let $f_Y:X_Y \rt Y $ be a minimal right $\CX$-approximation of $Y.$ Then $f_Y$ induces an isomorphism
\[ \Ext^i_{\La}(X, X_Y)\simeq \Ext^i_{\La}(X, Y)\]
for all $i\geq 1$ and $X \in \CX.$
\end{lemma}

\begin{theorem} \label{1-Gortrans}
Assume that $\La$ is a $1$-Gorenstein algebra. Let $M$ be a non-projective module  in $\CP^{<\infty}(\La).$ Then there is the following isomorphism in $\overline{\CP^{<\infty}(\La)}$
\[ \tau_{\CP^{<\infty}(\La)}(M)\simeq \overline{\rm{Coker}}(\rm{Mimo}(f))\]
where $P_1 \st{f} \rt P_0 \rt \tau_{\La}(M) \rt 0$ is a minimal projective presentation of $\tau_{\La}(M).$ In particular, for any $N \in \CP^{\infty}(\La)$, 
\[D\underline{\Hom}_{\La}(M, N)\simeq \Ext_{\La}^1(N, \overline{\rm{Coker}}(\rm{Mimo}(f))).\]
\end{theorem}
\begin{proof}

If we consider $(P_1 \st{f} \rt P_0)$ as  an object in $H(\La)$,  then in view of above facts, we have a minimal right  $\mathcal{S}(\La)$-approximation $(1_{P_1}, [1_{P_0},0]):\rm{Mimo}(f) \rt f$ in $H(\La).$ Then by taking cokernel, we reach morphism $h:\rm{Coker}(\rm{Mimo}(f)) \rt \tau_{\La}(M)$ in $\mmod \La.$ We first show that $h$ is a   right $\CP^{<\infty }(\La)$-approximation  of $\tau(M).$ Let $Z$ be in $\CP^{< \infty }(\La)$ and $g:Z \rt \tau(M)$ a morphism in $\mmod \La.$ We need to find $g':Z \rt \rm{Coker}(\rm{Mimo}(f))$ such that $h\circ g'=g.$ Let $Q_1 \st{d} \rt Q_0 \rt Z \rt 0$ be a minimal projective presentation of $Z$. Since $Z \in \CP^{< \infty}(\La)$, then it implies $d$ is a monomorphism. Note that since $\La$ is a $1$-Gorenstein algebra, then $\CP^{< \infty}(\La)$ is exactly modules with projective dimension less than two. By lifting property, we can obtain morphism $G$ from $(Q_1 \st{d}\rt Q_0)$ to $(P_1 \st{f} \rt P_0)$ in $H(\La).$ On the other hand, since $d \in \CS(\La)$, then there exists
morphism $G':d \rt \rm{Mimo}(f)$ such that $(1_{P_1}, [1_{P_0}, 0])\circ G'=G.$ Now, by taking cokernel, we get morphism $g':Z \rt \rm{Coker}(\rm{Mimo}(f))$ in $\mmod \La$ which satisfies the condition we need. Thus $h$ is a right $\CP^{< \infty }(\La)$-approximation of $\tau_{\La}(M)$ in $\mmod \La$. If we decompose $h$ as $h=[h_1, h_2]: \rm{Coker}(\rm{Mimo}(f))=N_1\oplus N_2 \rt \rm{Coker}(f)=\tau_{\La}(M)$ such that $h_1$ is right minimal and $h_2=0.$ By our notation $N_1=\overline{\rm{Coker}}(\rm{Mimo}(f))$. As $h$ is right $\CP^{< \infty }(\La)$-approximation, it follows that $h_1$ is minimal right $\CP^{< \infty }(\La)$-approximation. By Lemma \ref{Applic}, $h_1$ induces $\Ext_{\La}^1(X,\overline{\rm{Coker}}(\rm{Mimo}(f)) )\simeq \Ext_{\La}^1(X, \tau(M))$ $$ (\dagger) \ \  \ \ \ \ \Ext_{\La}^1(-,\overline{\rm{Coker}}(\rm{Mimo}(f)) )\mid_{\CX}\simeq \Ext_{\La}^1(-, \tau(M))\mid_{\CX}.$$

On the other hand, by Theorem \ref{AR-duality} we have
$$\Ext_{\La}^1(X, \tau(M))\simeq D\underline{\Hom}_{\La}(M, X)\simeq \Ext_{\La}^1(X, \tau_{\CP^{<\infty}(\La)}(M))$$
for any $X$ in $\CP^{<\infty}(\La).$ This implies 
$$(\dagger \dagger) \ \ \ \ \ \Ext^1(-, \overline{\rm{Coker}}(\rm{Mimo}(f)))\mid_{\CP^{<\infty}(\La)} \simeq \Ext_{\La}^1(-, \tau_{\CP^{<\infty}(\La)}(M))\mid_{\CP^{<\infty}(\La).}$$

By combing $(\dagger)$ and $(\dagger \dagger)$, we have $\Phi((\overline{\overline{\rm{Coker}}(\rm{Mimo}(f))}, -))\simeq \Phi((\overline{\tau_{\CP^{<\infty}(\La)}(M)}, -))$, see \ref{Firstconstuction} for definition of $\Phi$. The equivalence $\Phi$ implies $(\overline{\overline{\rm{Coker}}(\rm{Mimo}(f))}, -)\simeq (\overline{\tau_{\CP^{<\infty}(\La)}(M)}, -).$ But it means $\tau_{\CP^{<\infty}}(M)\simeq \overline{\rm{Coker}}(\rm{Mimo}(f))$ in $\overline{\CP^{<\infty}(\La)}$.
\end{proof}
  We close this section with some results concerning subcategories in Example \ref{Examplefirst} $(ii).$
\begin{lemma}\label{lemma4.5}
Assume $\CX \subseteq \mmod \La$ holding conditions of Proposition \ref{DualityDec} and take  $X \in \CX$. Then
\begin{itemize}
\item[(1)] If  $\Omega^{-1}\tau (X) \in \CX$ and whose projective cover in $\mmod \La$ is injective, then  $\Phi_1((\underline{X}, -))= (-,\underline{\Omega^{-1}\tau(X)}).$
\item[(2)] If projective cover of $X$ is injective and also $\tau^{-1}\Omega(X)$ belongs to $\CX$, then $\Phi^{-1}_1((-, \underline{X}))=(\underline{\rm{Tr}D\Omega(X)}, -).$
\end{itemize}  
\end{lemma}
\begin{proof}
Assume that  $Q$ is the projective cover of   $\Omega^{-1}\tau (X)$ in $\mmod \La$, that is, injective by our assumption. Then we obtain the following exact sequence
$$0 \rt \Tor_1(-, \rm{Tr}(X)) \rt -\otimes_{\La} D\Omega^{-1}D \rm{Tr}(X) \rt -\otimes_{\La}D(Q) \rt -\otimes_{\La}\rm{Tr}(X) \rt 0$$
in $\mmod (\rm{mod}\mbox{-} \Lambda)$. We use  $\Tor_1(-, M)\simeq \underline{\Hom}_{\La}(\rm{Tr}M, -)$, for each  left module $M $,  so
$$0 \rt \underline{\Hom}_{\La}(X, -) \rt -\otimes_{\La} D\Omega^{-1}D \rm{Tr}(X) \rt -\otimes_{\La}D(Q) \rt -\otimes_{\La}\rm{Tr}(X) \rt 0,$$
note that $\rm{Tr}\rm{Tr}(X)\simeq X.$  Then the  restriction of above exact sequence on $\CX$ gives us a similar one in $\mmod \CX.$ Consider $X \in \CX,$ such that  $\Omega^{-1}\tau(X)$  also belongs to $\CX$, then just by following the definition of the  duality $\Phi_1$, we have
$$\Phi_1((\underline{X}, -))=\Upsilon \circ \Psi' \circ (\Psi'_1)^{-1}((\underline{X}, -))=\Upsilon \circ \Psi'(D(Q) \rt \rm{Tr}(X))=\Upsilon((\overline{D\Omega^{-1}\tau(X)}, -))=(-,\underline{\Omega^{-1}\tau(X)}).$$  The same argument can work for $(2)$. 
\end{proof}

\begin{remark}\label{remark234}
Let $\CX$ be the same as Lemma \ref{lemma4.5}. If $X \in \CX$ holds the assumptions Lemma \ref{lemma4.5}$(1)$, then since   $\Phi_1$ is a duality, for such $X$, $(-, \underline{\Omega^{-1}\tau(X)})$ becomes an injective-projective object in $\mmod \underline{\CX}.$ Similarly, if  $ X$ satisfies Lemma \ref{lemma4.5}(2), then $\Phi^{-1}_1((-, \underline{X}))=(\underline{\rm{Tr}D\Omega(X)}, -)$ is an injective-projective object in $\underline{\CX} \mbox{-} \rm{mod}$.
In particular, if $\CX=\mmod \La$ for a self-injective $\La$, then automatically $\CX$ holds   all conditions we need  and then in view of these facts one can deduce that $\mmod (\underline{\rm{mod}}\mbox{-} \La)$ is a Frobenius category . Although, it can be seen by  a general fact, that is,  $\mmod \mathcal{T}$ of a triangulated category $\mathcal{T}$ is a Frobenius category. The fact is shown due to Freyd and Verdier independently in \cite{F2} and  \cite{V}.  Moreover, by Proposition \ref{DualityDec}, for each $X \in \mmod \La$
$$\Phi_1((\underline{X}, -)) \simeq D((\underline{X}, -))=(-, \underline{\Omega^{-1}\tau(X)}).$$
It means that $\Omega^{-1}\tau$ is a serre duality for $\underline{\rm{mod} } \mbox{-} \La $, see also \cite{ES} for another proof. Recall from \cite{RV} that a Serre functor for a $k$-linear triangulated category $\CT$ is  an auto-equivalence $S: \CT \rt \CT$ together with an isomorphism $D\Hom_{\CT}(X, -)\simeq \Hom_{\mathcal{T}}(-, S(X))$ for each $X \in \CT$, with $D$ is the duality $\Hom_{k}(-, E).$ It is interesting to note that here we didn't use Auslander-Reiten duality  to prove the existence serre duality as \cite{RV}.  

\end{remark}

Recall that a triangulated subcategory $\mathcal{S}$ of triangulated category $\mathcal{T}$ is thick provided that all idempotents split. 
\begin{proposition}
Let $\Lambda$ be a self-injective algebra and $\mathcal{Y}$ a thick subcategory of $\underline{\rm{mod}} \mbox{-}\Lambda$. If the primage $\pi^{-1}(\CY)$, see Example \ref{Examplefirst}, is functorially finite, e.g. $\La$ being  of finite type, then
\begin{itemize}
\item[(1)] $\CY$ has serre functor;
\item[(2)] For each $Y \in \CY$ there exists morphism $f:S_{\CY} \rt \Omega^{-1}\tau(Y)$ in $\underline{\rm{mod}} \mbox{-} \Lambda$ such that for every object $X \in \CY$ and each morphism $g:X \rt \Omega^{-1}\tau(Y)$ in $\underline{\rm{mod}} \mbox{-} \Lambda$  there exists unique morphism $h:X \rt \Omega^{-1}\tau(Y)$ with $g=fg.$ 
\end{itemize} 
\end{proposition} 
\begin{proof}
As we observed in Example \ref{Examplefirst}, $\pi^{-1}(\CY)$ satisfies the assumptions of Proposition \ref{DualityDec}, in particular the assumptions of Theorem \ref{StableDualizing}.  
But $\underline{\pi^{-1}(\CY)}=\CY$ is a dualizing $k$-variety by  Theorem \ref{StableDualizing}, and due to \cite[Corollary 2.6]{C1} has a serre duality, denote it by $S_{\CY}.$ Hence, for each $Y \in \CY$, by definition of serre duality, we have
$$(1) \ \  \ \ \  D\underline{\Hom}_{\La} (Y, -) \simeq \underline{\Hom}_{\La} (-, S_{\CY}(Y)).$$

On the other hand, as we discussed in Remark \ref{remark234}, we have
$$(2) \ \  \ \ \  D\underline{\Hom}_{\La} (Y, -) \simeq \underline{\Hom}_{\La} (-, \Omega^{-1}\tau (Y))\vert_{\CY}.$$

 So combining $(1)$ and $(2)$ imply that $\Theta: \underline{\Hom}_{\La}(-, S_{\CY}(Y))\simeq \underline{\Hom}_{\La}(-, \Omega^{-1}\tau(Y))\vert_{\CY}.$ It is enough to take $f$ to be the image of $1_{S_{\CY}(Y)}$ under isomorphism $$\Theta_{S_{\CY}(Y)}:\underline{\Hom}_{\La}(S_{\CY}(Y), S_{\CY}(Y))\simeq \underline{\Hom}_{\La}(S_{\CY}(Y), \Omega^{-1}\tau(Y))\vert_{\CY}.$$
\end{proof}

\section{ The Auslander-Reiten duality  for  Gorenstein projective modules  }
In this section we shall concentrate on the subcategory of Gorenstein projective modules. Throughout this section we assume that the subcategory $\rm{Gprj} \mbox{-} \La$ of Gorenstein projective modules  is contravariantly finite in $\mmod \La.$ For example for virtually Gorenstein algebras it happens, see \ref{gproj}. 

Recall that for an $\La$-module $M$, its syzygy $\Omega(M)$ is the kernel of its projective  cover $P(M) \st{p_{M}} \rt M.$ This gives rise to the syzygy functor $\Omega:\underline{\rm{mod}} \mbox{-} \La \rt \underline{\rm{mod}}\mbox{-} \La$. Then the syzygy functor $\Omega$ restricts to an auto-equivalence $\Omega: \underline{\rm{Gprj}} \mbox{-}\La \rt \underline{\rm{Gprj}} \mbox{-} \La$. As we mentioned in \ref{gproj}, the stable category $\underline{\rm{Gprj} }\mbox{-} \La$ becomes a triangulated category such that the translation functor is given by a
quasi-inverse of $\Omega$, and that the triangles are induced by the  short exact sequences in $\Gprj \La.$

Let us  first recall some basic definitions for almost split sequence, see \cite{ARS}  and \cite{AS} for more details. Let $\mathcal{C}$ be a subcategory of $\mmod \La$ closed under direct summands and extensions. A morphism $f:A \rt B$ in $\mathcal{C}$ is a left almost split morphism in $\mathcal{C}$ if it isn't a split monomorphism and every morphism $j:A \rt X$ in $\mathcal{C}$ that is not a split monomorphism factors through $f.$

  A right almost split morphism in $\mathcal{C}$ is defined by duality. An exact sequence $0 \rt A  \st{f} \rt B \st{g} \rt C \rt 0$ in $\mathcal{C}$ is said to be an almost split sequence in $\mathcal{C}$ if $f$ is a left almost split morphism in $\mathcal{C}$ and $g$ is a right almost split morphism in $\mathcal{C}.$ The indecomposable module $A$ is uniquely determined by $C$ and  denoted by $\sigma_{\mathcal{C}}(C).$ In the case that $\mathcal{C}$ is functorial finite, then $\sigma_{\mathcal{C}}(A)$ exists for non-$\rm{Ext}$-projective modules in $\mathcal{C}$. We say a module $C \in \mathcal{C}$ is  $\rm{Ext}$-projective in $\mathcal{C},$ if $\Ext_{\La}^1(C, X)=0$ for any $X$ in $\mathcal{C}.$  Note that for $\mathcal{C}=\rm{Gprj} \mbox{-} \La$, $\rm{Ext}$-projective modules are exactly projective modules in $\mmod \La.$  
  
In the sequel, for a triangulated category $\mathcal{T}$ being Hom-finite $k$-linear Krull-Schmidt, we have a triangulated version of the concept of almost split sequence in $\CT$ as follows:
Following \cite{H} a triangle $  A \st{f} \rt B \st{g}\rt C \st{h} \rt  A[1]$ in $\CT$ is called an Auslander-Reiten
triangle if the following conditions are satisfied:
\begin{itemize}
\item [$(\rm{AR}_1)$] $A$ and $C$ are indecomposable.
\item[$(\rm{AR}_2)$] $h\neq 0$.
\item [$(\rm{AR}_3) $]If $D$ is indecomposable, then for every non-isomorphism $ t : D \rt C $ we have
$ht = 0.$
\end{itemize}
\begin{theorem}\label{tranGprj}
Assume that $\rm{Gprj}\mbox{-} \La$ is contravariently finite subcategory in $\mmod \La.$ Then $\rm{Gprj}\mbox{-} \La$ has Auslander-Reiten duality,  i.e. there exist an equivalence functor $\tau_{\mathcal{G}}:\underline{\rm{Gprj}}\mbox{-} \La \rt \underline{\rm{Gprj}}\mbox{-} \La$  such that for any Gorenstein projective modules $G$ and $G'$, we have the following natural isomorphism in both variables;
\[D\underline{\Hom}_{\La}(G, G')\simeq \Ext_{\La}^1(G', \tau_{\mathcal{G}}(G)).\]

\end{theorem}
\begin{proof}
Since $\Gprj \La$ is a resolving subcategory of $\mmod \La$ then \cite[Corollary 0.3]{KS} follows that it is functorially finite. Due to \cite[Corollary 2.6]{C1} and Theorem \ref{StableDualizing}, $\underline{\rm{Gprj}}\mbox{-} \La$ has a serre duality, see Remark \ref{remark234} for definition of Serre duality, and we denote it by $S_{\mathcal{G}}$.
Therefore, for every $G$ and $G'$ in $\Gprj \La$
$$ (\dagger) \ \ \ \ \ D\underline{\Hom}_{\La}(G, G')\simeq \underline{\Hom}_{\La}(G', S_{\mathcal{G}}(G)).$$

Denote by $\tau_{\mathcal{G}}$  the equivalence $S_{\mathcal{G}} \circ \Omega$.
Let $G_1$ be an indecomposable non-projective Gorenstein projective module. Then by applying \cite[Theorem 1.1]{AS}, there exists an almost split sequence $0 \rt G_1' \rt G'' \rt G_1 \rt 0$, with $G_1'$ to be necessary an indecomposable  non- projective Gorenstein projective. It is straightforward to check that the induced triangle
 $$G'_1 \rt G'' \rt G_1 \rt \Omega^{-1}(G'),$$
 in $\underline{\rm{Gprj}} \mbox{-} \La$ is an Auslander-Reiten triangle. Moreover by \cite[Proposition I.2.3.]{RV},  $\tau_{\mathcal{G}}(G_1)\simeq G_1'$ in $\Gprj \La.$ In view of $(\dagger)$ and using the natural isomorphism $\Ext_{\La}^1(G', -)\vert_{\rm{Gprj} \mbox{-}\La}\simeq \underline{\Hom}_{\La}(\Omega(G'), -)$ for each $G'$ in $\rm{Gprj} \mbox{-} \La$, we can obtain the following relative Auslander-Reiten duality
 \[D\underline{\Hom}_{\La}(G_1, G')\simeq \Ext_{\La}^1(G', \tau_{\mathcal{G}}(G_1)).\]
for each  $G'$ in $\rm{Gprj} \mbox{-} \La.$ By the additivity of functor $\tau_{\mathcal{G}}$, and Krull-schmit property of $\rm{Gprj} \mbox{-} \La$, we can extend above isomorphism for each $G \in \rm{Gprj} \mbox{-} \La$. Thus, we get the desired isomorphism.
\end{proof}

By the proof of the above theorem, we have the following consequence.
\begin{corollary}\label{Galmost}
	Let $G$ be an indecomposable non-projective Gorensetin projective $\La$-module. Then $\sigma_{\rm{Gprj} \mbox{-} \La}(G)\simeq \tau_{\mathcal{G}}(G)$ in $\underline{\rm{Gprj}}\mbox{-} \La.$
\end{corollary}
\begin{proof}
	Note that as $\rm{Gprj} \mbox{-} \La$ is functorial finite by our assumption, then $\sigma_{\rm{Gprj} \mbox{-} \La}(G)$ always exists for any indecomposable non-projective Gorensetin projective module.  
\end{proof}
A relative version of Auslander's conjecture as explained in \ref{AusCOnjsec}  also holds for the subcategory of Gorenstein projective modules.
\begin{proposition}\label{AUSGPrj}
Let $G$ be in $\rm{Gprj} \mbox{-} \La$. 
\begin{itemize}
	\item[$(i)$]If $F$ is a direct summand of $\Ext_{\La}^1(G,-)\vert_{\rm{Gprj} \mbox{-} \La}$, then there exists $G' \in \rm{Gprj} \mbox{-} \La $ such that $F\simeq \Ext_{\La}^1(G', -)\vert_{\rm{Gprj} \mbox{-} \La}.$

\item[$(ii)$] If $F$ is a direct summand of $\Ext_{\La}^1(-, G)\vert_{\rm{Gprj} \mbox{-} \La}$, then there exists $G' \in \rm{Gprj} \mbox{-} \La $ such that $F\simeq \Ext_{\La}^1(-, G')\vert_{\rm{Gprj} \mbox{-} \La}.$
\end{itemize}
\end{proposition}
\begin{proof}
By the above Theorem we can deduce $\Ext_{\La}^1(G,-)\vert_{\rm{Gprj} \mbox{-} \La}\simeq D\underline{\Hom}_{\La}(-,\tau_{\mathcal{G}}(G))$. On the other hand, since $\underline{\rm{Gprj}} \mbox{-} \La$ is a dualizing $R$-variety, see Theorem \ref{StableDualizing},  and has split idempotents, then any injective functor in $\mmod \underline{\rm{Gprj}} \mbox{-} \La$ is in the form of $D\underline{\Hom}_{\La}(-,A)$ for some $A $ in $\rm{Gprj} \mbox{-} \La.$ So, the isomorphism implies that $\Ext_{\La}^1(G,-)$ is an injective functor in $\mmod \underline{\rm{Gprj}} \mbox{-} \La$ and also $F$ as a direct summand of it. Hence, there is $G' \in \rm{Gprj} \mbox{-} \La $ such that $F\simeq D\underline{\Hom}_{\La}(-,G')$, which by the above theorem $F\simeq \Ext_{\La}^1(\tau^{-1}_{\mathcal{G}}(G'),-)\vert_{\rm{Gprj} \mbox{-} \La}.$
\end{proof}
In the following we provide some examples in which the relative Ausalnder-Reiten translation in $\rm{Gprj} \mbox{-} \La$ for a given Gorenstein projective module are computed. 
\begin{example}\label{SgTran}
\begin{itemize}
\item[$(i)$] Let $\La$ be a self-injective algebra. Then the subcategory $\mathcal{S}(\La) $, because of Lemma \ref{GPCH}, is mapped to $\rm{Gprj}\mbox{-}  T_2(\La)$ by the equivalence $H(\La) \simeq \mmod T_2(\La)$. For an non-projective object $A \st{f}\rt B$ in $\mathcal{S}(\La)$, whose relative Auslander-Reiten translation in $\mathcal{S}(\La)$ is given by
$$\sigma_{\mathcal{S}(\La)}\simeq\tau_{\mathcal{S}(\La)}\simeq \rm{Mimo}(\tau_{\La}(B \rt \rm{Coker}(f))),$$
see \cite{RS3} for more details. In particular, by Theorem \ref{tranGprj}, for any object $C \st{g} \rt D$ in $\mathcal{S}(\La)$, we have 
 $$D\Ext_{H(\La)}^1(f, g)\simeq \underline{\Hom}_{H(\La)}(g, \rm{Mimo}(\tau_{\La}(B \rt \rm{Coker}(f))) ).$$ By the additivity of functors involved in above isomorphism, we can extend the isomorphism for all object $A \st{f}\rt B$ in $\mathcal{S}(\La)$, not to be necasserly indecomposable.
\item[$(ii)$] Let $\La$ be a $\rm{d}$-Gorenstein algebra. It was proved by Auslander and Reiten in \cite{AR4}, for indecomposable non-projective Gorenstein projective $C$, $\sigma_{\rm{Gprj} \mbox{-} \La}(C)\simeq \Omega^{d}D\Omega^{d}\rm{Tr}(C),$ in $\underline{\rm{Gprj}}\mbox{-} \La.$  In similar argument of the first example, we deduce the following  for every $G$ and $G'$ in $\Gprj \La$ 
 \[D\underline{\Hom}_{\La}(G, G')\simeq \Ext_{\La}^1(G',\Omega^{d}D\Omega^{d}\rm{Tr}(G) ).\]

\end{itemize}
\end{example}

 Let us recall that  an abelian category $\CA $ is said to be semisimple if any short exact sequence splits. It is equivalent to say that any object is projective provided $\CA$ with enough projectives.
\begin{lemma}\label{MinimalINJ} Let $G$ be in $\rm{Gprj} \mbox{-} \La$ and non-projective. Then the following exact sequence
\[0 \rt (-, \Omega(G)) \rt (-, P) \rt (-, G) \rt (-, \underline{G}) \rt 0\]
is a minimal projective resolution of $(-, \underline{G})$ in $\mmod (\rm{Gprj}\mbox{-}\La).$
\end{lemma}
\begin{proof}
In case that $G$ is indecomposable the $\Omega(G)$ is as well. By this fact one can deduce for the case of $G$ being indecomposable.  Then extending this to any non-projective Gorenstein projective module, by considering this fact that finite direct sums of minimal projective resolutions are again minimal.
\end{proof}
\begin{lemma}\label{indeCOM}
Let $G$ be a non-projective indecomposable module in $\Gprj \La$. Then representation 
$\Omega(G)\hookrightarrow P$ is an indecomposable Gorenstein projective object in $T_2(\Lambda)$, where $P$ is the projective cover of $G$.
\end{lemma}
\begin{proof}
Let $\phi=(\phi_1, \phi_2)$ is an endomorphism of $\Omega(G)\hookrightarrow P$. By the property of a  projective cover, we can conclude that $\phi$ is a non-isomorphism if and only if $\phi_2$ so is. In view of this fact and  being $\rm{End}(G)$  local, it follows that the set of all non-isomorphism endomorphism makes a two-sided ideal, or equivalently, $\rm{End}_{\La}(\Omega(G)\hookrightarrow P)$ is local. It implies that $\Omega(G)\hookrightarrow P$ is indecomposable.
\end{proof}
\begin{theorem}\label{CHArecterzationSemi}
 Consider the  following statements:
\begin{itemize}
\item[$(i)$] $ \mmod (\underline{\rm{Gprj}} \mbox{-} \Lambda)$ is semisimple;

\item[$(ii)$] Any indecomposable Gorenstein projective object of $H(\Lambda)$ is isomorphic to one of the following indecomposable Gorenstein projective representations:
\begin{itemize}
\item[$(a)$] $G \st{\rm{Id}} \rt G$, for some indecomposable  module $G \in \Gprj \La$,
\item[$(b)$] $0 \rt G$, for some indecomposable  module $G \in \Gprj \La$,
\item[$(c)$] $\Omega(G) \hookrightarrow P$, where $P$ is the projective cover of some indecomposable non-projective $G \in \Gprj \La;$
\end{itemize}
\item[$(iii)$]  The relative Auslander-Reiten translation  for indecomposable non-projective Gorenstein projective $G $ in the subcategory $\Gprj \Lambda$ is isomorphic to $\Omega(G),$ i.e. $\tau_{\mathcal{G}}(G)\simeq \Omega (G).$
\end{itemize}
Then $ (i)\Rightarrow (ii) \Rightarrow (iii).$ Moreover, if $\La$ is of finite $\rm{CM}$-type, then all the statements are equivalent.
\end{theorem}
\begin{proof}
$(i) \Longrightarrow (ii)$ The representations of types $(a)$ and $(b)$ are indecomposable, not difficult to prove, but of type  $(c)$ follows by Lemma \ref{indeCOM}. Conversely, assume $G \st{f} \rt G'$  is an indecomposable Gorenstein projective indecomposable representation on quiver $\mathbb{A}_2$. This representation induces the short exact sequence $0 \rt G \st{f} \rt G' \rt G'' \rt 0$ in $\Gprj \La.$  This short exact sequence gives us the following functor $F$
$$0 \rt (-, G) \rt (-, G') \rt (-, G'') \rt F \rt 0$$ in $\mmod (\underline{\rm{Gprj }} \mbox{-}\La)$. Since $ \mmod (\underline{\rm {Gprj}} \mbox{-} \Lambda)$ is semisimple, then $F$ is a projective functor, namely, $F\simeq (-, \underline{A})$ for some $A$ in $\Gprj \La.$ By Lemma \ref{MinimalINJ} we know that $0 \rt (-, \Omega(A)) \rt (-, P) \rt (-, A) \rt (-, \underline{A}) \rt 0$
is a minimal projective resolution of $ (-, \underline{A}) $. Hence, we have  two projective resolution for $F$ in $\mmod (\Gprj \La)$, by the property of  minimality, the minimal  one is a direct summand of the first one. Now  the Yoneda Lemma  says us the short exact sequence $0 \rt G \rt G' \rt G'' \rt 0$ is a direct sums of
$0 \rt \Omega(A) \rt P \rt A \rt 0 $, $0 \rt 0 \rt B \st{\rm{Id}} \rt B \rt 0$ and $0 \rt C \st{\rm{Id}} \rt C \rt 0 \rt 0$ for some $ A, B$ and $C$ in $\Gprj \La.$ But this implies that the indecomposable representation $f$ must be one of three types.\\
$(ii) \Longrightarrow (iii) $ In view of the   observation provided in the first of this section  we can conclude that there is  an almost split sequence $0 \rt \tau_{\mathcal{G}}(G) \rt G' \rt G \rt 0 $ in $\Gprj \La.$ Considering $\tau_{\mathcal{G}}(G) \rt G'$
 as a representation over quiver $\mathbb{A}_2$, then by $(ii)$ it  can be decomposed to the indecomposable representations in types of $(a), (b)$ and $(c)$. According to these facts that the end terms in an almost split sequence  are indecomposable and being non-split, then $\tau_{\mathcal{G}}(G) \rt G'$  must be an indecomposable representation in the form of $(c).$ It is our desired result.\\
Assume $\La$ is of  $\rm{CM}$-finite type, then $\mmod (\Gprj \La)$ is in fact equivalent to the category of finitely generated modules over some artin algebra. Now we prove  $(iii) \Longrightarrow (i)$ to complete our task.
Let $(-, \underline{G})$ be a indecomposable projective functor in $\mmod (\underline{\rm{Gprj}} \mbox{-}\La)$. Without loss of generality  we can assume that $G$ is an indecomposable non-projective  Gorenstein projective module. Then by $(iii)$ there is  an almost split $\eta: 0 \rt \Omega(G) \rt P \rt G \rt 0$ which gives us
$$0 \rt (- \Omega(G)) \rt (-, P) \rt (-, G) \rt (-, \underline{G}) \rt 0.$$ By the properties of an almost split sequence it is routine to check that $(-, \underline{G})$ is a simple functor in $\mmod (\Gprj \La)$, and also in $\mmod (\underline{\rm{Gprj}}\mbox{-} \La)$. Therefore, any projective functor is a simple functor, and hence $\mmod (\underline{\rm{Gprj}}\mbox{-} \La)$  is semisimple.
\end{proof}
\begin{corollary}
	Let $\La$ be an algebra satisfying the equivalent condition in above theorem. Then for any Gorenstein projective modules $G$ and $G'$, we have the following;
	\[D\underline{\Hom}_{\La}(G, G')\simeq \Ext_{\La}^1(G', \Omega(G)).\]
	 
\end{corollary}
Since $\mmod ( \underline{\rm{Gprj}} \mbox{-} \Lambda)$ is a Frobenius category, so global dimension  $\mmod (\underline{\rm{Gprj}} \mbox{-} \Lambda)$ is finite if and only if $\mmod (\underline{\rm{Gprj}} \mbox{-} \Lambda)$ is  a semisimple abelian category. So in above theorem we completely describe when $\mmod ( \underline{\rm{Gprj}} \mbox{-} \Lambda)$ has a finite global dimension.

\begin{remark} \label{RemaAUS}
Auslander and Reiten in \cite{AR1} investigated when the global dimension of  finitely presented module over additive category $\underline{\rm{mod}} \mbox{-} \La$, $\mmod (\underline{\rm{mod}} \mbox{-} \La)$, is zero, or equivalently, to be semisimple. In \cite [Theorem 10.7]{AR1} it is proved that global dimension of $\mmod (\underline{\rm{mod}} \mbox{-} \La)$ is zero if and only if $\La$ is a Nakayama algebra of loewy length less than of 3.  The overlap of their result with ours is when $\La$ to  be a self-injective algebra. Based on these results this question can be asked  whether $\La$ is of finite $\rm{CM}$-type when
 $ \mmod (\underline{\rm{Gprj}} \mbox{-} \Lambda)$ is semisimple.
\end{remark}

In the following let us list  some examples satisfying the conditions of Theorem \ref{CHArecterzationSemi}.
\begin{example}\label{ExamplSimi}
\begin{itemize}
\item[$(1)$] Let $\La$ be a self-injective Nakayama algebra with loewy length less than of 3, see Remark \ref{RemaAUS}.
\item[$(2)$] If $\La$ is of finite $\rm{CM}$-type and satisfies the equivalent conditions of Theorem \ref{CHArecterzationSemi}, then $T_2(\La)$ is as well. In particular, for this case $T_2(\La)$ is of finite $\rm{CM}$-type and whose relative Auslander-Reiten translation is again the first syzygy functor. By repeating this  process let $T^1_2(\La)=T_2(\La),$ $T^n_2(\La)=T_2(T^{n-1}_2(\La))$, we can get many algebras satisfying the conditions. 

\item[$(3)$] Let $\La = kQ/I$ be a monomial algebra. We say that the algebra $\La$ is quadratic monomial provided that the ideal $I$ is generated by paths of length two. In \cite[Theorem 5.7]{CSZ}, it is proved that there is a triangle equivalence
$$\underline{\rm{Gprj}} \mbox{-} \Lambda \simeq \CT_{d_1} \times \CT_{d_2} \times \cdots \times \CT_{d_m}$$
where $\CT_{d_n}=(\mmod k^n, \sigma^n)$ for some natural number $n,$ and auto-equivalence $\sigma_n: \mmod k^n \rt \mmod k^n.$ We recall from \cite[Lemma 3.4]{C2} that for a semisimple abelian category $\CA$ and an auto-equivalence $\Sigma$ on $\CA,$ there is a unique triangulated structure on $\CA$ with $\Sigma$ as the translation functor. So by the equivalence a quadratic monomial algebra holds the conditions, as for each $m$, $\mmod (\mmod k^m)$ is semisimple.
\item[$(4)$] Let $\La$ be a simple gluing algebra of $A$ and $B$. Please see the  papers \cite{Lu1} and \cite{Lu2} for two different ways of gluing algebras. In these two mentioned papers, for two types of gluing, is proved that $\underline{\rm{Gprj}} \mbox{-} \Lambda\simeq\underline{\rm{Gprj}} \mbox{-} A \coprod \underline{\rm{Gprj}}\mbox{-} B$. Therefore, if $\mmod (\underline{\rm{Gprj}} \mbox{-} A)$ and $\mmod (\underline{\rm{Gprj}} \mbox{-} B)$ are semisimple, then $\mmod( \underline{\rm{Gprj}} \mbox{-} \Lambda)$ so is. For example, cluster-tilted algebras of type $\mathbb{A}_n$ and endomorphism algebras of maximal rigid objects of cluster tube $\mathcal{C}_n$ can be built as simple gluing algebras satisfying the condition of being semisimple, see \cite[ Corollary 3.21 and 3.22]{Lu2} .
\item [$(5)$]By Theorem 3.24 in \cite{AHV}, if $\Lambda$ and $\Lambda'$ are derived equivalence, then $\underline{\rm{Gprj}} \mbox{-} \Lambda \simeq \underline{\rm{Gprj}} \mbox{-} \Lambda'.$ Therefore, the property of being semisimple can be preserved under derived equivalence.
\end{itemize}

\end{example}

\section*{Acknowledgments}
The most of this work was carried out during a post-doc visit of the author in the university of Picardie. The author was supported by  "bourse du gouvernement fran\c caisthe  embassy" for his visit.
He is deeply indebted to Professor Alexander Zimmermann for his kind hospitality, inspiration and continuous encouragement. The author  would like to thank  Professor Henning Krause for introducing him the reference \cite{GLS} which motivated Theorem \ref{1-Gortrans}. He is also thankful to Dr. Mohammad Hossein Keshavarz for helping him to present this manuscript  better.

\end{document}